\documentclass[12pt]{amsart}
\usepackage{a4wide}
\usepackage{amssymb,bm,mathrsfs}
\usepackage{amsthm}
\usepackage{amsmath}
\usepackage{amscd}
\usepackage{verbatim}
\usepackage[all]{xy}
\usepackage{hyperref}

\usepackage{bookmark}

\usepackage{enumerate}

\usepackage{xr}
\numberwithin{equation}{section}

\theoremstyle{plain}
\newtheorem{theorem}{Theorem}[section]
\newtheorem{corollary}[theorem]{Corollary}
\newtheorem{lemma}[theorem]{Lemma}
\newtheorem{proposition}[theorem]{Proposition}

\theoremstyle{definition}
\newtheorem{definition}[theorem]{Definition}

\theoremstyle{remark}
\newtheorem{remark}[theorem]{Remark}

\renewcommand{\P}{\mathbb{P}}

\newcommand{\C}{\mathbb{C}}
\newcommand{\Q}{\mathbb{Q}}
\newcommand{\R}{\mathbb{R}}
\newcommand{\N}{\mathbb{N}}
\newcommand{\Z}{\mathbb{Z}}
\renewcommand{\H}{\mathbb{H}}
\renewcommand{\O}{\mathcal{O}}

\newcommand{\I}{\mathcal{I}}

\renewcommand{\a}{\mathfrak{a}}
\renewcommand{\b}{\mathfrak{b}}
\renewcommand{\d}{\mathfrak{d}}

\newcommand{\ep}{\varepsilon}
\newcommand{\ph}{\varphi}

\newcommand{\set}[1]{\left\{ #1 \right\}}
\newcommand{\setpm}{\set{\pm 1}}
\newcommand{\setm}[1]{\setminus\set{ #1 }}
\newcommand{\br}[1]{\left( #1 \right)}
\newcommand{\sq}[1]{\left[ #1 \right]}
\newcommand{\abs}[1]{\left|#1\right|}

\newcommand{\bs}{\backslash}

\newcommand{\se}{\subset}

\newcommand{\defC}{:\;}

\newcommand{\OK}{\O_K}

\newcommand{\IK}{\mathcal{I}_K}

\newcommand{\SLK}{\SL_2(K)}

\newcommand{\PSLK}{\PSL_2(K)}

\newcommand{\GLR}{\GL_2(\R)}

\newcommand{\GLRp}{\GL_2^+(\R)}

\newcommand{\SLp}[1]{\SL(\OK \oplus #1)}
\newcommand{\SLa}{\SLp{\a}}
\newcommand{\SLb}{\SLp{\b}}
\newcommand{\SLM}[2]{M(#1,#2)}
\newcommand{\Ga}{\Gamma_\a}
\newcommand{\Gb}{\Gamma_\b}
\newcommand{\Gas}[1]{\Gamma_{\a,#1}}
\newcommand{\Ginf}{\Gas{\infty}}
\newcommand{\Ginfb}{\overline{\Ginf}}

\newcommand{\Xa}{X(\a)}

\newcommand{\XB}[1]{X(#1)^*}
\newcommand{\XaB}{\XB{\a}}
\newcommand{\XbB}{\XB{\b}}
\newcommand{\XaH}{\overline{X(\a)}}

\renewcommand{\matrix}[4]{\begin{pmatrix}#1 & #2\\#3 & #4 \end{pmatrix}}
\newcommand{\sMatrix}[4]{\br{\begin{smallmatrix}#1 & #2\\#3 & #4 \end{smallmatrix}}}
\newcommand{\abcd}{\matrix{a}{b}{c}{d}}

\newcommand{\matrixCol}[2]{\begin{pmatrix}#1 \\ #2\end{pmatrix}}

\newcommand{\del}{\partial}	

\newcommand{\zb}{\overline{z}}
\newcommand{\ub}{\overline{u}}
\newcommand{\vb}{\overline{v}}

\newcommand{\hF}{{}_2F_1}

\DeclareMathOperator{\tr}{tr}
\DeclareMathOperator{\vol}{vol}

\DeclareMathOperator{\SL}{SL}
\DeclareMathOperator{\PSL}{PSL}
\DeclareMathOperator{\GL}{GL}

\DeclareMathOperator{\OGroup}{O}
\DeclareMathOperator{\UGroup}{U}

\DeclareMathOperator{\sgn}{sgn}

\DeclareMathOperator{\Beta}{B}

\let\div\relax
\DeclareMathOperator{\div}{div}

\newcommand{\equiArrow}{\quad\Leftrightarrow\quad}
\newcommand{\equiDefArrow}{\quad:\Leftrightarrow\quad}
\newcommand*{\isoArrow}{\mathrel{\vcenter{\offinterlineskip\hbox{\begin{small}$\mkern6mu\sim$\end{small}}\hbox{$\longrightarrow$}}}}

\newcommand{\und}{\quad\text{and}\quad}  \newcommand{\with}{\quad\text{with}\quad}

\newcommand{\ssum}{\sideset{}{'}\sum}

\begin{document}

\title[Automorphic Green functions on Hilbert modular surfaces]{Automorphic Green functions\\on Hilbert modular surfaces}

\author{Johannes J. Buck}

\address{Fachbereich Mathematik, Technische Universit\"at Darmstadt, Schlossgartenstrasse 7,
D--64289 Darmstadt, Germany}
\email{jbuck@mathematik.tu-darmstadt.de}

\thanks{The author was supported by the DFG Collaborative Research Centre TRR 326 Geometry and Arithmetic of Uniformized Structures, project number 444845124.}

\date{\today}

\begin{abstract}
In this paper, we generalize results of Bruinier on automorphic Green functions on Hilbert modular surfaces to arbitrary ideals. For instance, we compute the Fourier expansion of the unregularized Green functions, use it to regularize them, obtain the Fourier expansion of the regularized Green functions and evaluate integrals of unregularized and regularized Green functions. Furthermore, we investigate their growth behavior at the cusps in the Hirzebruch compactification by computing the precise vanishing orders along the exceptional divisors. This makes the arithmetic Hirzebruch--Zagier theorem from Bruinier, Burgos Gil and Kühn more explicit. To this end, we generalize the theory of local Borcherds products. Lastly, we investigate a new decomposition of the Green functions into smooth functions and compute and estimate the Fourier coefficients of those smooth functions.
Finally, this is employed to prove the well-definedness and almost everywhere convergence of the generating series of the Green functions and the modularity of its integral.
\end{abstract}

\maketitle

\setcounter{tocdepth}{1}
\tableofcontents
\setcounter{tocdepth}{2}

\section{Introduction}

In 1976, Hirzebruch and Zagier showed that the intersection numbers of Hirzebruch--Zagier divisors on Hilbert modular surfaces can be interpreted as the Fourier coefficients of holomorphic elliptic modular forms of weight~$2$ (cf.~\cite{hirzebruch1976intersection}).
This result can essentially be reformulated by stating that the generating series
\[
A(\tau) = c_1(\mathcal M_{-1/2}(\C)) + \sum_{m=1}^\infty Z(m) q^m \in \Q[[q]] \otimes_\Q \operatorname{CH}^1(\overline{X})_\Q
\]
is a holomorphic modular form of weight~$2$, level $D$ and nebentypus $\chi_D$ with values in $\operatorname{CH}^1(\overline{X})_\Q$.
Here, by $D$ we denote the discriminant of the underlying real quadratic number field $K$, by $c_1(\mathcal M_{k}(\C))$ the first Chern class of the line bundle of modular forms of weight~$k$, by $\overline{X}$ the Hirzebruch compactification of the Hilbert modular surface $X$ associated to $K$ and by $Z(m)$ certain extensions of the Hirzebruch--Zagier divisors $T(m)$ of discriminant $m$ on $X$ to the Hirzebruch compactification $\overline{X}$.

Kudla and Millson aimed at a generalization of this result and studied special cycles for the orthogonal group $\OGroup(p,q)$ and the unitary group $\UGroup(p,q)$ in great generality by means of the Weil representation (cf.~\cite{kudlamillson}).
In the Kudla program one is interested in having arithmetic analogues to the Hirzebruch--Zagier theorem (cf.~\cite{kudla2002derivatives} and \cite{kudla2004special}).
More precisely, instead of proving the modularity of generating series like $A(\tau)$ with coefficients in classical Chow groups, one is interested in proving the modularity of generating series with coefficients in arithmetic Chow groups. The elements of arithmetic Chow groups are arithmetic divisors (or more general arithmetic cycles) up to rational equivalence. An arithmetic divisor in turn is a pair $(Z,g)$, where $Z$ is a classical divisor (on an integral model of $\overline X$) and $g$ is a Green current corresponding to $Z$.

Particular cases to study the Kudla program are smooth compactifications of Hilbert modular surfaces. On them there are two natural choices to complete the Hirzebruch--Zagier divisors with Green currents to arithmetic divisors. The first such choice is given by the automorphic Green functions introduced by Bruinier in \cite{bruinier1999borcherds} and the second by Kudla Green functions (cf.~\cite{kudla1997central} and \cite{kudla2004special}).
The author dealt in his dissertation \cite{buckdiss} with both types. In this paper we confine ourselves to automorphic Green functions and present many results of his thesis with slight extensions. In an upcoming paper we will deal with the Kudla Green functions.

Automorphic Green functions on Hilbert modular surfaces were investigated earlier in works of Bruinier and Bruinier, Burgos Gil and Kühn under certain assumptions on the level and the discriminant (cf.~\cite{bruinier1999borcherds} and \cite{bruinier2007borcherds}).
In the present paper, we provide extensions of their results and add new results.
The first generalization is that we associate to each fractional ideal $\a \in \IK$ its Hilbert modular group $\Ga$ and corresponding Hilbert modular surfaces $\Xa$, $\XaB$ and $\XaH$ with automorphic Green functions $\Phi(\a,m,s,z)$ and its regularization $\Phi(\a,m,z)$. Classically, mainly the case $\a = \OK$ was considered. However, this generalization is necessary to investigate the classical Green function $\Phi(\OK,m,z)$ near a cusp $\kappa \in \P^1(K)$, since this corresponds to the investigation of the Green function $\Phi(\a^2,m,z)$ near the cusp~$\infty$ for $\a \in \IK$ chosen appropriately.

After discussing some essentials of Hilbert modular groups, associated lattices, Hilbert modular surfaces with their Hirzebruch--Zagier divisors and pre-log-log Green functions in the sense of \cite{BKK1} in Section~\ref{pre-section}, we start Section~\ref{auto-section} with the computation of the Fourier expansions of the unregularized Green function $\Phi(\a,m,s,z)$.
This allows us to identify a Dirichlet series of representation numbers associated to the ideal $\a$ in the Fourier expansion which is responsible for the diverging behavior of $\Phi(\a,m,s,z)$ at the harmonic point $s=1$. We then describe a regularization process and obtain the regularized Green function $\Phi(\a,m,z)$ with its Fourier expansion following the basic argument of \cite{bruinier1999borcherds} and \cite{zagier1975modular}. Here, a closer investigation of the part of the Fourier expansion which generates the logarithmic singularities along the Hirzebruch--Zagier divisors near the cusp~$\infty$ takes place.
After developing a more general theory of local Borcherds products of \cite{bruinier2001local} and \cite[p. 150--153]{123mf} in Subsection~\ref{local-bp-section}, we are able to identify local Borcherds products in the Fourier expansion of $\Phi(\a,m,z)$ and we are able to describe the vanishing order of those products at the exceptional divisors over the cusps. By our definition of the Hirzebruch--Zagier divisors $Z(\a,m)$ on the Hirzebruch compactification $\XaH$ the vanishing orders coincide with the respective multiplicities.
This together with some estimates of the remaining terms in the Fourier expansion proves our first main result and closes Section~\ref{auto-section}.
\begin{theorem}[cf.~Theorem \ref{Phi-is-green}]
The function $\Phi(\a,m,z)$ is a pre-log-log Green function on $\XaH$ with respect to the divisor $Z(\a,m)$.
\end{theorem}

\sloppy In Section~\ref{decomp-section} we present a new decomposition of the unregularized Green function $\Phi(\a,m,s,z)$ into smooth $\Ga$ invariant functions $\Phi_n(\a,m,s,z)$ for $n \in \N_0$, which induces a respective decomposition of the regularized Green function $\Phi(\a,m,z)$ into smooth $\Ga$ invariant functions as well. Because of the smoothness of the functions $\Phi_n(\a,m,s,z)$ they possess an everywhere converging Fourier expansion.
We partially compute the Fourier coefficients of this expansion and estimate the remaining Fourier coefficients which allows us to show the integrability of the regularized Green function $\Phi(\a,m,z)$ and to obtain a polynomial bound in $m$ on the integral of the absolute value $\abs{\Phi(\a,m,z)}$ of the Green function in Section~\ref{integral-section}. This last section deals with integrability and the actual integrals of the unregularized and regularized Green functions and the components of their smooth decomposition. All the integrals can be made explicit in term of volumes of Hirzebruch--Zagier divisors, for instance we show
\[
\int_{\Xa} \Phi(\a,m,s,z) \omega^2 = \frac{2 \vol(T(\a,m))}{s(s-1)}
\und
\int_{\Xa} \Phi(\a,m,z) \omega^2 = -2 \vol(T(\a,m)).
\]
in Theorem~\ref{Phi-s-integral} and Theorem~\ref{Phi-integral} for $m\in \N$ and $\Re(s)>1$.
We then employ the polynomial growth in $m$ of the integrals of $\abs{\Phi(\a,m,z)}$ to derive the following striking consequence.
\begin{theorem}[cf.~Corollary~\ref{almost-everywhere-coro}]
The generating series
\begin{align} \label{Phi-gen-series} 
\sum_{m=1}^\infty \Phi(\a,m,z) q^m
\end{align}
with $q \in \C$, $|q|<1$ converges absolutely for almost all $z \in \H^2$ and is integrable over $\Xa$.
\end{theorem}
Note that in the Kudla program one deals with arithmetic generating series where the coefficients are arithmetic divisors interpreted as elements of the first arithmetic Chow group. Here however, we make sense of the generating series over the actual Green functions $\Phi(\a,m,z)$. The pointwise limit which exists for almost all $z \in \Xa$ gives rise to a current because of its integrability. However, as function on $\Xa$ the limit is almost nowhere continuous (cf.~Remark~\ref{almost-nowhere-continuous}).

Finally, we compute the integral of the generating series~\eqref{Phi-gen-series} over $\Xa$ and prove its modularity in Theorem~\ref{modular-integral}.

\subsection*{Acknowledgements}

The author thanks his doctoral advisor Jan H. Bruinier for his support.
Since many results of this paper are part of the author's dissertation, the more detailed acknowledgements of \cite{buckdiss} apply here as well.

\section{Preliminaries} \label{pre-section} 

\subsection{The underlying real quadratic number field}

Throughout the paper $K$ is a real quadratic field of discriminant $D$. With $x \mapsto x'$ we denote the conjugation in $K$, with $N(x):=xx'$ and $\tr(x):=x+x'$ the norm and the trace. The trace is a $\Q$ linear map and the norm is a non-degenerate quadratic form turning $K$ into a rational quadratic space of signature $(1,1)$. Another non-degenerate quadratic form is induced by the trace $(x,y)\mapsto \tr(xy)$. The latter is positive definite, i.e., of signature $(2,0)$.
The ring of integers of $K$ is given by
\[
\OK = \Z + \tfrac{D+\sqrt{D}}{2} \Z.
\]
By Dirichlet's unit theorem, there exists a unique $\ep_0>1$ (we understand $K$ as subfield of $\R$ with $\sqrt{D}>0$) such that
\[\OK^\times = \set{\pm \ep_0^k \defC k \in \Z}. \]

Analogously, there  exists a unique $\ep_1>1$ such that
\[ \OK^+ := \OK^\times \cap K^+ = \set{ \ep_1^k \defC k \in \Z}
\with
K^+:=\set{x \in K \defC x \gg 0}
. \]
Here  $x \gg 0$ being totally positive means $x>0$ and $x'>0$. If $N(\ep_0)=1$, we have $\ep_1= \ep_0$ and otherwise $\ep_1 = \ep_0^2$.

By $\IK$ we denote the ideal group of $K$.
Recall that two ideals $\a, \b \in \IK$ belong to the same genus if and only if there exists a $\lambda \in K$ with $N(\lambda) N(\a)=N(\b)$.
The dual $\a^\vee$ of an ideal $\a \in \IK$ with respect to the norm form is given by ${(\a\d)'}^{-1}$ and with respect to the trace form by $(\a\d)^{-1}$.
Here, $\d$ denotes the \emph{different} $\d=(\sqrt{D}) = \sqrt{D}\OK$.
The volume of $\a$ is given by $\vol(\a)=N(\a)\sqrt{D}$ with respect to both forms.

\subsection{Hilbert modular groups} 
In this paper we consider the Hilbert modular groups
\[
\Ga := \SLa := \matrix{\OK}{\a^{-1}}{\a}{\OK} \cap \SLK
\]
associated to $\a \in \IK$.
Recall that they act by
\[
\abcd (\alpha : \beta) := (a\alpha +b \beta : c\alpha + d \beta)
\]
on $\P^1(K)$. The quotient of this operation defines the cusps of $\Ga$ of which there are $h_K$ many with $h_K$ being the class number of $K$.
For relations between different Hilbert modular groups and lattices associated to them it is useful to introduce the sets 
\begin{align} \label{SLM-def} 
\SLM{\a}{\b} := \matrix{\a}{(\a\b)^{-1}}{\a\b}{\a^{-1}} \cap \SLK
\end{align}
for $\a,\b \in \IK$. They satisfy the equations
\begin{align}
\SLM{\a}{\b}^{-1} = \SLM{\a^{-1}}{\a^2\b}
\und
\SLM{\a_1}{\b} \SLM{\a_2}{\a_1^2 \b} = \SLM{\a_1 \a_2}{\b}.
\end{align}
For example they imply
\[
M^{-1} \SLb M = \SLp{\a^2\b}.
\]
for all $M \in \SLM{\a}{\b}$. Hence, the Hilbert modular groups $\Ga$ and $\Gb$ are conjugated if $\a\b$ is a square in the group $\IK$.

\subsection{Lattices associated to ideals}
Throughout the paper $V$ denotes the $\Q$ vector space
\begin{align} \label{def-V} 
V := \set{ \matrix{a}{\lambda'}{\lambda}{b} \in K^{2\times 2} \defC a,b \in \Q,\lambda  \in K } = \set{A \in K^{2\times 2} \defC A^\top = A'}.
\end{align}
equipped with the determinant as quadratic form of signature $(2,2)$.
Using the map
\[
\SLK \to \OGroup(V),
\quad
M \mapsto (A \mapsto M.A := M A (M')^\top)
\]
whose kernal is given by $\pm 1$ we can view $\PSLK := \SLK / \set{\pm 1}$ as subgroup of $\OGroup(V)$.

We associate to each $\a \in \IK$ the lattice
\[
L(\a) := \set{ \matrix{a}{\lambda'}{\lambda}{b} \in V \defC a \in \Z, b \in N(\a)\Z,\lambda  \in \a }.
\]
Its dual is given by
\begin{align} \label{La-dual} 
L(\a)^\vee = \set{ \frac{1}{N(\a)}\matrix{a}{\lambda'}{\lambda}{b} \in V \defC a \in\Z, b \in  N(\a)\Z,\lambda  \in \a\d^{-1} }
\end{align}
and we have
\begin{align} \label{lattice-translate} 
M.L(\a^2\b) = N(\a) L(\b)
\und
M.L(\a^2\b)^\vee = \frac{L(\b)^\vee}{N(\a)}
\end{align}
for all $\a, \b \in \IK$ and $M \in \SLM{\a}{\b}$.
In particular, the lattices $L(\a)$ and $L(\a)^\vee$ are invariant under $\Ga$.

With $\H := \set{z \in \C \defC \Im(z)>0}$ being the complex upper half plane every point $z \in \H^2$ gives rise to an orthogonal decomposition $W_z \oplus \tilde W_z$ of $V_\R := V \otimes_\Z \R$ such that the quadratic form (the determinant) restricted to $W_z$ is negative definite and the determinant restricted to $\tilde W_z$ is positive definite. Namely, the vectors
\[
X_z := \matrix{x_1x_2-y_1y_2}{x_1}{x_2}{1}
,\quad
Y_z := \matrix{x_1y_2+x_2y_1}{y_1}{y_2}{0}
\]
and
\[
\tilde X_z := \matrix{x_1x_2+y_1y_2}{x_1}{x_2}{1}
,\quad
\tilde Y_z := \matrix{x_1y_2-x_2y_1}{-y_1}{y_2}{0}
\]
with $z = (z_1,z_2)=(x_1+iy_1,x_2+iy_2)$
form an orthogonal basis of $V_\R$. The first two span $W_z$ and the last two $\tilde W_z$.
We obtain a decomposition $\det=q_{W_z}+q_{\tilde W_z}$ with $q_{W_z}$ being the projection onto $W_z$ composed with the determinant. For later use we define $h(A,z) := -q_{W_z}(A)$ and obtain the majorant
\[
q_z(A) := h(A,z) + q_{\tilde W_z}(A) = \det(A)+2h(A,z),
\]
a positive definite quadratic form on $V_\R$.
For elements $A=\sMatrix{a}{\lambda'}{\lambda}{b}\in V_\R$ we obtain
\begin{align} \label{hA-frac} 
h(A,z) = \frac{|b z_1 z_2  - \lambda z_1 - \lambda' z_2 +a |^2}{4 y_1y_2}
\und
q_{\tilde W_z}(A)  = \frac{|b \zb_1 z_2  - \lambda \zb_1 - \lambda' z_2 +a |^2}{4 y_1y_2}.
\end{align}
For anisotropic $A=\sMatrix{a}{\lambda'}{\lambda}{b}\in V_\R$ the normalized function
\[
g(A,z) := \frac{h(A,z)}{\det(A)} = \frac{\abs{bz_1z_2 - \lambda z_1 - \lambda' z_2 + a}^2}{4 y_1y_2 \det(A)}
\]
comes in handy from time to time. We have
\begin{align} \label{hg-transform} 
h(A,z) = h(M.A,Mz)
\und
g(A,z) = g(M.A,Mz)
\end{align}
for all $M \in \SLK$.
\begin{remark} \label{gA-with-d-remark} 
Using the $\GLR$ invariant hyperbolic distance $d(z_1,z_2) := |z_1-z_2|^2/y_1y_2$ we can express $g(A,z)$ with $S := \sMatrix{0}{-1}{1}{0}$ by
\[
g(A,z) = \frac{d(z_1,ASz_2)}{4} = \frac{|z_1-ASz_2|^2}{4\Im(z_1)\Im(SAz_2)}.
\]
\end{remark}

\subsection{Hilbert modular surfaces and Hirzebruch--Zagier divisors}
We associate to each $\a \in \IK$ its Hilbert modular surface
$\Xa := \Ga \bs \H^2$. By $\XaB := X(\a) \cup \Ga \bs \P^1(K)$ we denote its Baily--Borel compactification (cf.~\cite[Chapter II, Section 1.2]{123mf} or \cite[Chapter I, Section 2]{freitag1990hilbert} for an introduction), a normal complex space.
The cusps in $\XaB$ are highly singular but can be desingularized; one obtains the Hirzebruch compactification $\XaH$ (cf.~\cite[Chapter II]{geer1988hms}) which is smooth at the boundary (the only left over singular points are the elliptic fix points which are finite quotient singularities). In the Hirzebruch compactification every cusp $\kappa$ is replaced by an exceptional divisor $E^\kappa(\a)$ which consists of finitely many glued $S_k \cong \P^1(\C)$ ($k \in \Z/r_\kappa\Z$ for $r_\kappa \in \N$ chosen appropriately for each cusp $\kappa$).
We call the sum of all exceptional divisors $E(\a) := \sum_{\kappa \in \Ga \bs \P^1(K)} E^\kappa(\a)$.

Let $\b \in \IK$ and $\kappa = (\alpha : \beta) \in \P^1(K)$. Then there exists a matrix $M \in \SLM{\a}{\b}$ with $\a := \alpha \OK + \beta \b^{-1}$ and $M \infty = \kappa$. Now the map
\begin{align} \label{M-iso} 
(\H^2)^* \to (\H^2)^*,\quad z \mapsto M^{-1}z
\end{align}
induces an isomorphism $\XbB \isoArrow \XB{\a^2\b}$ mapping the cusp $\kappa$ of $\XbB$ to the cusp~$\infty$ of $\XB{\a^2\b}$. That is why it is enough to study the cusp~$\infty$ for all $\Xa$ instead of all cusps of $X(\OK)$. To study the desingularized cusp~$\infty$ we have to express it in local coordinates. We call them $(u,v) \in \C^2$ and they satisfy
\begin{align} \label{local-to-z} 
\matrixCol{2\pi i z_1}{2\pi i z_2} = \matrix{\alpha}{\beta}{\alpha'}{\beta'}\matrixCol{\log u}{\log v}
\end{align}
with respect to a totally positive basis $(\alpha,\beta)$ of $\a^{-1}$. The $S_k$ correspond then (up to one point) to $u=0$ ($v=0$, respectively) and we have the following lemma.
\begin{lemma} \label{exponentials-in-local} 
Let $\nu \in \a\d^{-1}$. Then the following functions are $\a^{-1}$ invariant and can be expressed in local coordinates $(u,v)$ with respect to $(\alpha,\beta)$:
\begin{align*}
e(\tr(\nu z)) = e(\nu z_1)e(\nu' z_2) &= u^{\tr(\alpha \nu)} v^{\tr(\beta \nu)}, \\
e(\tr(\nu \overline{z})) = e(\nu \overline{z_1})e(\nu' \overline{z_2}) &= \overline{u}^{-\tr(\alpha \nu)} \overline{v}^{-\tr(\beta \nu)}, \\
e(\nu z_1)e(\nu' \overline{z_2})
&= u^{\alpha \nu} \overline{u}^{-\alpha'\nu'} v^{\beta \nu} \overline{v}^{-\beta'\nu'},\\
e(\nu \overline{z_1})e(\nu' z_2)
& = u^{\alpha' \nu'} \overline{u}^{-\alpha\nu} v^{\beta' \nu'} \overline{v}^{-\beta\nu}.
\end{align*}
The evaluation of the third and fourth line is independent of the chosen branch of the logarithm $\log(u)$ as long as the branch of $\log(\overline{u})$ is chosen accordingly, i.e., $\log(\overline{u}):=\overline{\log(u)}$. The same holds for $\log(v)$ and $\log(\overline{v})$, respectively.
\end{lemma}

As Kähler manifold $\Xa$ possesses a Kähler form
\begin{align} \label{omega-def} 
\omega := \eta_1+\eta_2
\with
\eta_j := \frac{1}{4 \pi} \frac{dx_j dy_j}{y_j^2}.
\end{align}
It induces the volume form $\omega^2$ which allows us to integrate over the Hilbert modular surface. For instance, its volume is given by $\vol(X(\a)) = \zeta_K(-1) = L(-1,\chi_D) \zeta(-1)$.

For non-zero $A=\sMatrix{a}{\lambda'}{\lambda}{b}\in V$ we define
\[
T_A := \set{ z \in \H^2 \defC h(A,z)=0 } = \set{ z \in \H^2 \defC bz_1z_2 - \lambda z_1 - \lambda' z_2 +a=0 }
\]
leading us for $m \in \N$ to the definition of the Hirzebruch--Zagier divisors
\begin{align} \label{Tam-def} 
T(\a,m) := \sum_{ \substack{A \in L(\a)^\vee / \set{\pm 1}\\ \det(A)=m/(N(\a)D)}  } T_A.
\end{align}
Because of the transformation law~\eqref{hg-transform} $T(\a,m)$ is invariant under $\Ga$ and therefore it is well-defined on $\Xa$.
By an argument similar to \cite[Section 3.2]{bruinier2007borcherds} it can be shown that
\begin{align} \label{Tm-volume-sum} 
\vol(T(\a,m))
= \sum_{ \substack{A \in \Ga \bs L(\a)^\vee / \set{\pm 1}\\ \det(A)=m/(N(\a)D)}  } \vol(T_A)
= \sum_{ \substack{A \in \Ga \bs L(\a)^\vee / \set{\pm 1}\\ \det(A)=m/(N(\a)D)}  } \vol( \Gas{\pm A}' \bs \H )
\end{align}
with respect to the pullback of the Kähler form $\omega$.
Here, the stabilizer $\Gas{ \pm A}'$ is given by
\[
\Gas{ \pm A}' := \set{ M' \defC M \in \Ga \und M.A \in \set{\pm A} }.
\]
A component $T_A \se \H^2$ with $A = \sMatrix{a}{\lambda'}{\lambda}{b}$ runs into the cusp~$\infty$ if and only if $b=0$.
Therefore, we define
\begin{align*}
\Lambda(\a,m) := &\set{\lambda \in \a\d^{-1} \defC N(\lambda) = - \frac{mN(\a)}{D} },\ 
\Lambda^\pm(\a,m) := &\set{\lambda \in \Lambda(\a,m)  \defC \sgn(\lambda)= \pm 1}
\end{align*}
and
\begin{align} \label{Tinf-def} 
T^\infty(\a,m) := \sum_{ \substack{A=\sMatrix{a}{\lambda'}{\lambda}{0} \in L(\a)^\vee / \set{\pm 1}\\ \det(A)=m/(N(\a)D)}  } T_A
=\sum_{\lambda \in \Lambda^+(\a,m)} \sum_{a\in \Z} \set{ z \in \H^2 \defC \tr(\lambda z)=a}.
\end{align}
The divisor $T^\infty(\a,m)$ is invariant under
\[
\Ginf = \set{ \matrix{\ep}{\mu}{0}{\ep^{-1}} \defC \ep \in \OK^\times, \mu \in \a^{-1} }.
\]
Now, consider the real codimension one submanifold
\[
S(\a,m) := \bigcup_{\lambda \in \Lambda^+(\a,m)} S_\lambda
\with
S_\lambda := \set{z \in \H^2 \defC \tr(\lambda y)=0}
\]
which contains the divisor $T^\infty(\a,m)$ as subset.
The connected components of its complement $\H^2 \setminus S(\a,m)$ are the so-called \emph{Weyl chambers of index $m$}.
The cyclic group $(\OK^\times)^2$ acts on $\Lambda^+(\a,m)$ by multiplication. The quotient is finite and with respect to a fixed $w \in (\R^+)^2$ it is possible to specify a unique representative from each orbit. Namely, we call $\lambda \in \Lambda^+(\a,m)$ \emph{reduced} with respect to $w$ if $\lambda$ is minimal with $\tr(\lambda w) \ge 0$ in its $(\OK^\times)^2$ orbit. The set of all reduced $\lambda \in \Lambda^+(\a,m)$ with respect to $w$ is denoted by $R(\a,m,w)$.
For all $w=(y_1,y_2)$ with $z \in W$ in a fixed Weyl chamber $W$ the set $R(\a,m,w)$ is the same. Therefore, we allow the notation $R(\a,m,W)$ with $W$ being a Weyl chamber.
We define
\[
\rho(\a,m,w) := \sum_{ \lambda \in  R(\a,m,w)} \frac{\lambda}{\ep_0^{2}-1}
\]
and call it \emph{Weyl vector} with respect to $w$ ($\rho(\a,m,W)$, respectively). 

Those Weyl vectors allow us now to define the completed Hirzebruch--Zagier divisors $Z(\a,m)$ for the Hirzebruch compactification $\XaH$.
Namely, the component of $Z(\a,m)$ living on $E^\infty(\a)$ is given by
\begin{align} \label{Zinf-def} 
Z^\infty(\a,m) := \sum_{k=1}^{r_\infty} \tr(\rho(\a,m,A_k)A_k) S_k.
\end{align}
Here, $A_k$ is the element of $\a^{-1}$ corresponding to $S_k$ due to \cite{geer1988hms}.
For other cusps than $\infty$, the divisor $Z^\kappa(\a,m)$ is given by the image of $Z^\infty(\a\b^2,m)$ under the isomorphism $\overline{X(\a\b^2)} \isoArrow \XaH$ (here $\b \in \IK$ has to be chosen appropriately, cf.~\eqref{M-iso}).
We then define
\[
Z(\a,m) := T(\a,m) + \sum_{\kappa \in \Ga \bs \P^1(K)} Z^\kappa(\a,m)
\]
where $\kappa$ runs through all cusps of $\Ga$.

\subsection{Logarithmic singularities and pre-log-log Green functions} 

We expect the reader to be familiar with the concept of logarithmic singularities and pre-log-log Green functions (cf.~\cite[Section 1.2]{bruinier2007borcherds} for details).
We use the following scaling of logarithmic singularities in this paper: Let $h$ be a holomorphic function. Then the function $\log(|h|^2)$ has logarithmic singularities along the divisor $\div(h)$. This scaling implies the satisfaction of the Green equation
\[
dd^c [g] + \delta_{Z} = [dd^c g]
\]
for a function $g$ with logarithmic singularities along the divisor $-Z$. Note that a Green function for a divisor $Z$ has logarithmic singularities along $-Z$ (and not $+Z$).

In our situation, the logarithmic singularities of the Green functions we investigate are along the Hirzebruch--Zagier divisors $-Z(\a,m)$. Additionally, we allow the pre-log-log growth along the exceptional divisor $E(\a)$.
\begin{remark} \label{pre-log-log-condition} 
Recall that a function $f$ is a pre-log-log growth form if and only if
\[
f,\quad
w_1 \log(|w_1|) \frac{\del f}{\del w_1},\quad
w_1  w_2 \log(|w_1|)\log(|w_2|) \frac{\del^2 f}{\del w_1\del w_2}
\]
have log-log growth for $w_1,w_2 \in \set{z_1,\dots,z_k,\zb_1,\dots,\zb_k}$ and $w_1 \ne w_2$.
In most cases where we have to prove a function $f$ to be a pre-log-log growth form we will see that the given terms even go to $0$ for small $z_i$ ($1 \le i \le k$).
\end{remark}

\section{Investigation of automorphic Green functions} \label{auto-section} 

\subsection{Unregularized Fourier expansion} \label{fourier-exp-and-reg} 
In this subsection we generalize the definition and many results of the automorphic Green function living on $X(\OK)$ from \cite{bruinier1999borcherds} to automorphic Green function living on $X(\a)$ for arbitrary ideals $a \in \IK$. We follow \cite[Section 3]{bruinier1999borcherds} and skip most of the proofs since they are similar to the proofs of the source.
\begin{definition} \label{Phi-amsz-def} 
For $\a \in \IK$, $m \in \N$, $s \in \C$ with $\Re(s)>1$ and $z \in \H^2 \setminus T(\a,m)$ we define
\[
\Phi(\a,m,s,z) :=
\sum_{ \substack{ A \in L(\a)^\vee \\ \det(A) = m/(N(\a)D) }  } Q_{s-1} \br{ 1+ 2 g(A,z) }.
\]
\end{definition}

Analogous to \cite{bruinier1999borcherds} we are rather interested in $\Phi(\a,m,s,z)$ at the harmonic point $s=1$, but unfortunately the series diverges at $s=1$. Therefore, we regularize $\Phi(\a,m,s,z)$ at $s=1$. For this purpose we have to compute the Fourier expansion of $\Phi(\a,m,s,z)$ and extend the Fourier coefficients meromorphically in $s$.

\begin{proposition} \label{first-Phi-prop} 
The series defining $\Phi(\a,m,s,z)$ converges normally for $\Re(s)>1$ and $z \in \H^2 \setminus T(\a,m)$ to a smooth function in $z$ which is $\Ga$ invariant and holomorphic in $s$. It is an eigenfunction of the Laplacian $\Delta_j$ for $j \in \set{1,2}$ with eigenvalue $s(s-1)$.
Further, we have for $\a,\b \in \IK$ and $M \in \SLM{\a}{\b}$
\[ \Phi(\b,m,s,Mz) = \Phi(\a^2\b,m,s,z). \]
\end{proposition}
\begin{proof}
The transformation law involving the ideals $\a,\b \in \IK$ implies the $\Ga$ invariance of $\Phi(\a,m,s,z)$ (choose $\a= \OK$, then $\SLM{\a}{\b} = \Gb$). The transformation law is a consequence of \eqref{lattice-translate} and \eqref{hg-transform}. For the rest follow \cite{bruinier1999borcherds}.
\end{proof}

Using the decomposition
\[
\Phi(\a,m,s,z)
= \sum_{b \in \Z} \Phi^b(\a,m,s,z)
= \Phi^0(\a,m,s,z)  + 2\sum_{b=1}^\infty \Phi^b(\a,m,s,z)
\]
with
\[
\Phi^b(\a,m,s,z) :=\sum_{ \substack{ A =\sMatrix{a}{\lambda'}{\lambda}{b} \in L(\a)^\vee \\ \det(A) = m/(N(\a)D) }  } Q_{s-1} \br{ 1+ 2 g(A,z) }
\]
we can compute the Fourier expansion of $\Phi(\a,m,s,z)$. Note that each single $\Phi^b(\a,m,s,z)$ is still invariant under $\Ginf$, in particular under translation by $\a^{-1}$ and possesses therefore a Fourier expansion.
Furthermore, for $b \in \N$ the function $\Phi^b(\a,m,s,z)$ has no singularity for arguments $z \in \H^2$ with $\Im(z) > B := m/(N(\a)Db^2)$.
For $b \in \N$ let $R^b(\a,m)$ be a set of representatives of
\[
\set{\lambda \in \a\d^{-1} / b\a \defC \frac{N(\sqrt{D}\lambda)}{N(\a)} \equiv m \pmod{bD}}.
\]
We compute for $b \in \N$
\begin{align*}
\Phi^b(\a,m,s,z)
&= \sum_{ \substack{a \in \Z/N(\a),\:\lambda \in \a\d^{-1}/N(\a) \\ ab-N(\lambda) = m/(N(\a)D) } }  Q_{s-1} \br{1+ \frac{|bz_1z_2-\lambda z_1-\lambda'z_2+a|^2}{2y_1y_2 m/(N(\a)D)} }\\
&= \sum_{ \substack{a \in \Z/N(\a),\:\lambda \in \a\d^{-1}/N(\a) \\ ab-N(\lambda) = m/(N(\a)D) } }  Q_{s-1}
\br{1+ \frac{| (z_1-\lambda'/b)(z_2-\lambda/b) +B |^2}{2y_1y_2 B} }\\
&= \sum_{\lambda \in R^b(\a,m)} \sum_{\mu \in \a}
Q_{s-1} \br{1+ \frac{ \abs{ \br {z_1- \frac{\lambda'+b\mu'}{N(\a)b}}\br{z_2 - \frac{\lambda+b\mu}{N(\a)b}} + B }^2}{2y_1y_2 B} }\\
&= \sum_{\lambda \in R^b(\a,m)} \sum_{\mu \in \a^{-1}}
Q_{s-1} \br{1+ \frac{ \abs{ \br {z_1 + \mu +  \frac{\lambda'}{N(\a)b}} \br {z_2 + \mu' +  \frac{\lambda}{N(\a)b}} + B }^2}{2y_1y_2 B} }.
\end{align*}
Hence, the problem is deduced to computing the Fourier expansion of the $\a^{-1}$ periodic function $H_s^B (\a^{-1},z)$ with
\begin{align} \label{HB-def}
H_s^B (\b,z) := \sum_{\mu \in \b} Q_{s-1}\br{ 1 + \frac{|(z_1+\mu)(z_2+\mu')+B|^2}{2y_1y_2B} }.
\end{align}
Namely, let
\[
H_s^B (\b,z) = \sum_{\nu \in (\b\d)^{-1} } b_s^B(\b,\nu,y) e(\tr(\nu x))
\]
be the Fourier expansion of $H_s^B (\b,z)$. Then we have
\begin{align*}
\Phi^b(\a,m,s,z)
&= \sum_{\lambda \in R^b(\a,m)} \sum_{\nu \in \a\d^{-1}} b_s^B(\a^{-1},\nu,y)  e \br{\tr  \br{\nu \br{ x +  \tfrac{\lambda'}{N(\a)b} }} }\\
&= \sum_{\nu \in \a\d^{-1}} \br{  \sum_{\lambda \in R^b(\a,m)} e \br{\tr  \br{\tfrac{\nu\lambda'}{N(\a)b} } }  }  b_s^B(\a^{-1},\nu,y) e(\tr(\nu x))\\
&= \sum_{\nu \in \a\d^{-1}} G^b(\a,m,\nu)  b_s^B(\a^{-1},\nu,y) e(\tr(\nu x))
\end{align*}
with the finite exponential sum
\begin{align} \label{gaus-sum} 
G^b(\a,m,\nu) := \sum_{ \substack{ \lambda \in \a\d^{-1} / b\a \\ \frac{N(\lambda)}{N(\a)} \equiv -\frac{m}{D}\: (b\Z) }  } e \br{\tr  \br{\tfrac{\nu\lambda'}{N(\a)b} } } .
\end{align}
\begin{definition} \label{our-bessel} 
For shorter notation we define for $\nu \in K^\times$
\[
\I^\nu_\kappa (z) := 
\begin{cases}
I_\kappa(z),&\quad N(\nu)>0,\\
J_\kappa(z),&\quad N(\nu)<0.
\end{cases}
\]
Here, $I_\kappa(z)$ and $J_\kappa(z)$ denote the respective Bessel functions, i.e.,
$I_\kappa(z)$ is the modified Bessel function of the first kind (cf.~\cite[10.25.2]{handbook})
and $J_\kappa(z)$ is the Bessel function of the first kind (cf.~\cite[10.2.2]{handbook}).
By $K_\kappa(z)$ we denote the modified Bessel function of the second kind (cf.~\cite[10.25.3]{handbook}).
\end{definition}

\begin{lemma} \label{Phi-b-fc} 
Let $\b \in \IK$ and $B>0$.
The function $H_s^B (\b,z)$ defined by equation~\eqref{HB-def} converges normally for $\Re(s)>1/2$ and for those $z \in \H^2$ at which no term in the series has a singularity, i.e., the arguments of all $Q_{s-1}$ are greater than $1$. For $y_1y_2>B$ this is the case and the series has the Fourier expansion
\[
H_s^B (\b,z) = \sum_{\nu \in (\b\d)^{-1} } b_s^B(\b,\nu,y) e(\tr(\nu x))
\]
with
\begin{align*}
b_s^B(\b,0,y) = &\frac{\pi \Gamma(s-1/2)^2}{2 \sqrt{D} N(\b) \Gamma(2s)} (4B)^s (y_1y_2)^{1-s},\\
b_s^B(\b,\nu,y) = &\frac{4 \pi}{N(\b)} \sqrt{ \frac{B y_1y_2}{D} } \I^\nu_{2s-1} \br{4 \pi \sqrt{B|N(\nu)|}} K_{s-1/2}(2\pi |\nu| y_1)\\ &\times K_{s-1/2}(2\pi |\nu'| y_2), \quad \text{if } \nu \ne 0.\\
\end{align*}
\end{lemma}
We are left with the analysis of $\Phi^0(\a,m,s,z)$. We have
\begin{align*}
\Phi^0(\a,m,s,z)
&= \sum_{ \substack{ A =\sMatrix{a}{\lambda'}{\lambda}{0} \in L(\a)^\vee \\ \det(A) = m/(N(\a)D) }  } Q_{s-1} \br{ 1+ 2 g(A,z) }\\
&= \sum_{ \substack{a \in \Z/N(\a),\:\lambda \in \a\d^{-1}/N(\a) \\ -N(\lambda) = m/(N(\a)D) } }  Q_{s-1} \br{1+ \frac{|-\lambda z_1-\lambda'z_2+a|^2}{2y_1y_2 m/(N(\a)D)} }\\
&= 2 \sum_{ \lambda \in \Lambda^+(\a,m) }
\sum_{a \in \Z}  Q_{s-1} \br{1+ \frac{|\lambda z_1+\lambda'z_2+a|^2}{2y_1y_2 mN(\a)/D} }.
\end{align*}
Let us define for $r_1,r_2 \in \R$
\begin{align} \label{alpha-beta-def}
\alpha(r_1,r_2) := \max(|r_1|,|r_2|)
\und
\beta(r_1,r_2) := \min(|r_1|,|r_2|).
\end{align}
\begin{lemma} \label{Phi-0-fc} 
The series
\begin{align*}
\Phi^0(\a,m,s,z)
= 2 \sum_{ \lambda \in \Lambda^+(\a,m) } \sum_{a \in \Z}  Q_{s-1} \br{1+ \frac{|\lambda z_1+\lambda'z_2+a|^2}{2y_1y_2 mN(\a)/D} }
\end{align*}
converges normally for $z \in \H^2 \setminus T^\infty(\a,m)$ and $\Re(s)>1/2$. Moreover, on $\H^2 \setminus S(\a,m)$ one has the Fourier expansion
\begin{align*}
\Phi^0(\a,m,s,z) =  \frac{4\pi}{2s-1}  \sum_{ \lambda \in \Lambda^+(\a,m) } \alpha(\lambda y_1,\lambda' y_2)^{1-s}\beta(\lambda y_1,\lambda' y_2)^s\\
+\: 4 \pi  \sum_{ \lambda \in \Lambda^+(\a,m) } \sum_{n=1}^\infty \sqrt{ |\lambda \lambda' y_1y_2| } I_{s-1/2} (2 \pi n \beta(\lambda y_1,\lambda'y_2))\\
\times \ K_{s-1/2}(2\pi n \alpha(\lambda y_1,\lambda' y_2)) \br{e(n \tr(\lambda x))+e(-n \tr(\lambda x))}.
\end{align*}
\end{lemma}

Hence, we obtain analogous to \cite{bruinier1999borcherds} that the individual $\Phi^b(\a,m,s,z)$ are well-defined for $\Re(s)>1/2$.

\subsection{Regularization with Fourier expansion}

Following \cite{buckdiss} one now investigates the ingredients of the Fourier expansion of $\Phi(\a,m,s,z)$ and finds that the only part which does not converge at $s=1$ is the component
\begin{align} \label{diverging-part} 
\frac{\pi \Gamma(s-1/2)^2}{\sqrt{D} \Gamma(2s)} \br{4m/D}^s  (N(\a)y_1y_2)^{1-s} \sum_{b=1}^\infty G^b(\a,m,0) b^{-2s}
\end{align}
of the constant Fourier coefficient. This is due to the diverging series $\sum_{b=1}^\infty G^b(\a,m,0) b^{-2s}$ of representation numbers. In \cite{bruinier2007borcherds} this series is investigated in more detail in the special case of $\a= \OK$, $D$ prime and $m \in \N$.
Analogously (with a lot of tedious computations which can be found in \cite{buckRepNumbers}) one shows Proposition~\ref{Gseries-with-div-sum} for which we need a generalized definition of the divisor sum:
\begin{definition} \label{div-sum-def} 
For odd discriminant $D$, $m \in \Z \setm{0}$ and $\a \se \OK$ such that $N(\a)$ is coprime to $D$ we define
\[
\sigma(\a,m,s) = |m|^{(1-s)/2} \sum_{d \mid m} d^s \prod_{p \mid D} (\chi_{D(p)}(d) + \chi_{D(p)}(N(\a)m/d)).
\]
The product ranges over all prime divisors $p$ of $D$ and $D(p) \in \set{\pm p}$ such that $D(p)$ is a discriminant, i.e. $D(p) \equiv 1 \pmod 4$.
Now for arbitrary $\b \in \IK$ we define
$\sigma(\b,m,s) := \sigma(\a,m,s)$
with $\a \se \OK$ coprime to $D$ within the genus of $\b$.
\end{definition}
The divisor sum satisfies the functional equation $\sigma(\a,m,s)=\sigma(\a,m,-s)$ (cf.~\cite[Lemma~7.3]{buckRepNumbers}).
\begin{proposition}[{{\cite[Theorem~8.1]{buckRepNumbers}}}] \label{Gseries-with-div-sum} 
For odd discriminant $D$ and $m \in \Z \setm{0}$ the series $\sum_{b=1}^\infty G^b(\a,m,0) b^{-s}$
has a meromorphic continuation to the complex plane with a simple pole at $s=2$. It satisfies
\[
\sum_{b=1}^\infty G^b(\a,m,0)b^{-s} = |m|^{-s/2} \frac{\zeta(s-1)}{L(s,\chi_D)} \sigma(\a,m,1-s).
\]
\end{proposition}
Analogous to \cite[eq. (2.39)]{bruinier2007borcherds} one defines\footnote{Note that the definition of $\ph(\a,m,s)$ and $q(\a,m)$ depends only on the genus of the ideal $\a$ whereas $L(\a,m)$ depends on the ideal $\a$ itself.} 
\[
\ph(\a,m,s) := -\frac{\Gamma(s-1/2)}{\Gamma(3/2-s)} \frac{\sigma(\a,m,1-2s)}{L(1-2s,\chi_D)}
\]
and proves using Proposition~\ref{Gseries-with-div-sum} that the constant term in the Laurent expansion at $s=1$ of \eqref{diverging-part} equals
\[
L(\a,m) - q(\a,m)\log(16 \pi^2 y_1y_2)
\]
with
\begin{align} \label{q-sigma} 
q(\a,m) = \ph(\a,m,1) = -\frac{\sigma(\a,m,-1)}{L(-1,\chi_D)}
\end{align}
being the residue at $s=1$ of \eqref{diverging-part} and
\begin{align} \label{L-constant} 
L(\a,m) = \ph(\a,m,1) \br{ 2 \frac{L'(-1,\chi_D)}{L(-1,\chi_D)} - 2 \frac{\sigma'(\a,m,-1)}{\sigma(\a,m,-1)} + \log\br{ \frac{D}{N(\a)} }}.
\end{align}
As direct consequence we obtain the following growth estimates.
\begin{corollary} \label{qL-growth} 
For large $m$ we have
\[
q(\a,m) = O(m^2)
\und
L(\a,m) = O(m^2 \log(m)).
\]
\end{corollary}

The understanding of the series~\eqref{diverging-part} leads then to the next Theorem.
\begin{theorem} \label{Phi-cont-theorem} 
The function $\Phi(\a,m,s,z)$ has a meromorphic continuation in $s$ to\\$\set{s \in \C \defC \Re(s)>3/4}$ for all $z \in \H^2 \setminus T(\a,m)$. Up to a simple pole at $s=1$ of residue $q(\a,m)$ it is holomorphic in this domain.
\end{theorem}

Theorem~\ref{Phi-cont-theorem} allows us to define the regularized automorphic Green function $\Phi(\a,m,z)$.
\begin{definition} \label{regularized-Green-def} 
We define
\[
\Phi(\a,m,z):=\mathcal{C}_{s=1} \sq{\Phi(\a,m,s,z)}
\]
to be the constant term in the Laurent expansion of $\Phi(\a,m,s,z)$ at $s=1$.
\end{definition}
By construction $\Phi(\a,m,z)$ is $\Ga$ invariant and has logarithmic singularities along $-T(\a,m)$ because of
\begin{align*}
Q_{0} \br{ 1+ 2 g(A,z) }
&= \frac{1}{2}\log\br{ \frac{g(A,z)+1}{g(A,z)} }
= \frac{1}{2}\log\br{ \frac{\det(A)+h(A,z)}{h(A,z)} }\\
&= \frac{1}{2}\log\br{ \frac{q_{\tilde W_z}(A)}{h(A,z)} }
= \log \abs{ \frac{bz_1\overline{z_2}-\lambda z_1 -\lambda' \overline{z_2} + a}{bz_1z_2 - \lambda z_1 - \lambda' z_2 + a} }
\end{align*}
for $A =\sMatrix{a}{\lambda'}{\lambda}{b}$ by \eqref{hA-frac}. The numerator is smooth and zero free on $\H^2$ while the denominator is holomorphic with the appropriate zero set. Recall that $- \log |bz_1z_2 - \lambda z_1 - \lambda' z_2 + a|$ occurs twice because $A$ comes together with $-A$.

\begin{theorem} \label{first-Phi-fc}
The Fourier expansion of $\Phi(\a,m,z)$ is given for $z \in \H^2 \setminus S(\a,m)$ with $\Im(z)>m/(DN(\a))$ by
\begin{align*}
&\Phi(\a,m,z)
= L(\a,m) - q(\a,m)\log(16 \pi^2 y_1y_2)\\
+\ &4 \pi \sum_{ \lambda \in \Lambda^+(\a,m) }\beta(\lambda y_1,\lambda' y_2)\\
+\ &  \sum_{ \lambda \in \Lambda^+(\a,m) } \sum_{n=1}^\infty \frac{e^{-2 \pi n |\tr(\lambda y)|}-e^{-2 \pi n ( \lambda y_1 - \lambda'y_2)}}{ n} \br{e(n \tr(\lambda x))+e(-n \tr(\lambda x))}\\
+\ &  \sum_{\substack{\nu \in \a\d^{-1} \\ \nu \ne 0 }} \frac{2 \pi}{D} \sqrt{ \frac{mN(\a)}{|N(\nu)|} } \exp(-2\pi \tr(|\nu|y))
  \sum_{b=1}^\infty  \frac{G^b(\a,m,\nu)}{b} \I^\nu_{1} \br{ \frac{4 \pi}{b} \sqrt{ \frac{m |N(\nu)|}{N(\a)D}}} e(\tr(\nu x)).
\end{align*}
\end{theorem}
\begin{proof}
The theorem is obtained from the preceding treatment of \eqref{diverging-part} together with an evaluation of the other terms of the Fourier expansion of $\Phi(\a,m,s,z)$ at $s=1$.
\end{proof}

\begin{lemma} \label{series-into-log} 
We have for $z \in \H^2 \setminus S(\a,m)$
\begin{align*}
\sum_{ \lambda \in \Lambda^+(\a,m) } \sum_{n=1}^\infty \frac{e^{-2 \pi n |\tr(\lambda y)|}-e^{-2 \pi n ( \lambda y_1 - \lambda'y_2)}}{ n} \br{e(n \tr(\lambda x))+e(-n \tr(\lambda x))}\\
= -4 \pi  \sum_{ \lambda \in \Lambda^+(\a,m)} \beta(\lambda y_1,\lambda' y_2)
+ 2\log \prod_{\lambda \in \Lambda^+(\a,m)} \abs{ \frac{1- e(|\lambda| z_1) \overline{ e(|\lambda'| z_2)}}{e(|\lambda| z_1) -e(|\lambda'| z_2)} }.
\end{align*}
\end{lemma}
\begin{proof}
This identity is proved by making use of the power series of the logarithm $\log(x)$ at $x=1$ and by careful case distinctions based on the sign of $\tr(\lambda y)$.
\end{proof}

Lemma~\ref{series-into-log} gives rise to the following simplification of Theorem~\ref{first-Phi-fc}.
\begin{theorem} \label{nice-Phi-rep} 
The Green function $\Phi(\a,m,z)$ is given for ${z \in \H^2 \setminus S(\a,m)}$ with $\Im(z)>m/(DN(\a))$ by
\begin{align*}
\Phi(\a,m,z)
&= L(\a,m) - q(\a,m)\log(16 \pi^2 y_1y_2)\\
&+\ 2\log \prod_{\lambda \in \Lambda^+(\a,m)} \abs{ \frac{1- e(|\lambda| z_1) \overline{ e(|\lambda'| z_2)}}{e(|\lambda| z_1) -e(|\lambda'| z_2)} }\\
&+\   \sum_{\substack{\nu \in \a\d^{-1} \\ \nu \gg 0 }} \frac{2 \pi}{D} \sqrt{ \frac{mN(\a)}{|N(\nu)|} } \sum_{b=1}^\infty  \frac{G^b(\a,m,\nu)}{b} I_{1} \br{ \frac{4 \pi}{b} \sqrt{ \frac{m |N(\nu)|}{N(\a)D}}}\\
&\times \ \br{e(\tr(\nu z)) + \overline{e(\tr(\nu z))}}\\
&+\ \sum_{\substack{\nu \in \a\d^{-1} \\ \nu >0, \, \nu'<0 }} \frac{2 \pi}{D} \sqrt{ \frac{mN(\a)}{|N(\nu)|} } \sum_{b=1}^\infty  \frac{G^b(\a,m,\nu)}{b} J_{1} \br{ \frac{4 \pi}{b} \sqrt{ \frac{m |N(\nu)|}{N(\a)D}}}\\
&\times \  \br{ e(\nu z_1)\overline{ e(-\nu'z_2) } + \overline{e(\nu z_1)} e(-\nu'z_2) }.
\end{align*}
\end{theorem}
\begin{proof}
Starting from Theorem~\ref{first-Phi-fc}, the main work was done in Lemma~\ref{series-into-log}. For the different notation of the exponentials in the last lines verify
for $\nu \in K^\times$
\begin{align*}
e (\tr(\nu x))  e(i \tr(|\nu|y)) =
\begin{cases}
e(\tr(\nu z)),&\quad \nu \gg 0,\\
\overline{e(-\tr(\nu z))},&\quad \nu \ll 0,\\
e(\nu z_1)\overline{ e(-\nu'z_2) },&\quad \nu>0,\ \nu'<0,\\
\overline{e(-\nu z_1)}e(\nu'z_2),&\quad \nu<0,\ \nu'>0.
\end{cases}
\end{align*}
Finally, by definition~\eqref{gaus-sum} of the exponential sum $G^b(\a,m,\nu)$ we have
\[
G^b(\a,m,\nu)=G^b(\a,m,-\nu)
\]
since the index set of the sum is invariant under multiplication with $-1$.
\end{proof}

\begin{proposition} \label{Phi-real-analytic} 
The regularized Green function $\Phi(\a,m,z)$ is real analytic and satisfies for $j \in \set{1,2}$
\[
\Delta_j \Phi(\a,m,z) = q(\a,m).
\]
\end{proposition}
\begin{proof}
In Theorem~\ref{nice-Phi-rep} we see that all terms except for $- q(\a,m)\log(16 \pi^2 y_1y_2)$ are the real part of a holomorphic function (in $z_1$ or $z_2$, respectively). This proves that $\Phi(\a,m,z)$ is real analytic and the Laplace equation follows with
\[
\Delta_j \log(y_1y_2) = \Delta_j \log(y_j) = -1.
\]
\end{proof}

\subsection{Local Borcherds product} \label{local-bp-section} 

In this subsection, we define for each ideal $\a \in \IK$ the local Borcherds product $\Psi(\a,m,z)$ at infinity in $\XaH$, obtain interesting representations and express it in local coordinates to determine its vanishing orders along the components of the exceptional divisor $E^\infty(\a)$.
The motivation is that the logarithmic singularities of $\Phi(\a,m,z)$ at and near infinity match (up to a factor) the logarithm of $|\Psi(\a,m,z)|$.
The latter is analyzed in Corollary~\ref{Bp-log-sing}.

\begin{definition} \label{Psi-sigma} 
Let
\[
\sigma : \Lambda^+(\a,m) \to \set{\pm 1}
\]
be a sign function with
\[
\lim_{\lambda \to 0} \sigma(\lambda) = +1
\und
\lim_{\lambda \to \infty} \sigma(\lambda) = -1.
\]
We define for $z \in \H^2$
\[
\Psi_\sigma(\a,m,z)
:= \prod_{\lambda \in \Lambda^+(\a,m)} \sigma(\lambda) \psi_\lambda(z)
\with
\psi_\lambda(z) :=  e(|\lambda| z_1) -e(|\lambda'| z_2).
\]
\end{definition}
\begin{remark}
The function $\sigma$ in the definition of $\Psi_\sigma(\a,m,z)$ is there for technical reasons only to make the product convergent. Namely, for fixed $z \in \H^2$ we have 
\[
\lim_{\lambda \to 0}\psi_\lambda(z) = +1
\und
\lim_{\lambda \to \infty }\psi_\lambda(z) = -1.
\]
By the equivalence relation
\[
\sigma_1 \sim \sigma_2
\equiDefArrow
\prod_{\lambda \in \Lambda^+(\a,m)} \sigma_1(\lambda)\sigma_2(\lambda) = 1
\]
we partition the set of all admissible sign functions $\sigma$ into two classes. Note that the product defining the equivalence relation is well-defined since almost all factors are equal to $1$.
We have
\[
\Psi_{\sigma_1}(\a,m,z) = \Psi_{\sigma_2}(\a,m,z) \equiArrow \sigma_1 \sim \sigma_2
\]
and
\[
\Psi_{\sigma_1}(\a,m,z) = -\Psi_{\sigma_2}(\a,m,z) \equiArrow \sigma_1 \not\sim \sigma_2.
\]
There is no canonical choice for the sign function $\sigma$, that is why we have to include it in the definition of $\Psi_\sigma(\a,m,z)$. Later we are mostly interested in $|\Psi(\a,m,z)|$ where the original sign of the product does not matter anymore. Whenever the sign is unimportant we simply write $\Psi(\a,m,z)$.
\end{remark}

\begin{proposition} \label{power-of-BP-is-invariant} 
The product $\Psi_\sigma(\a,m,z)$ is a holomorphic function on $\H^2$ with simple roots at $T^\infty(\a,m)$. Let $n \in 2\N$ with
\[
\frac{n}{1-\ep_0^2} \in \OK.
\]
Then $\Psi(\a,m,z)^n$ is invariant under $\Ginf$.
\end{proposition}
\begin{proof}
Clearly, each $\psi_\lambda(z)$ for $\lambda \in \Lambda^+(\a,m)$ is holomorphic. Consider
\begin{align*}
\psi_\lambda(z) = 0
&\equiArrow e(\lambda z_1) = e(-\lambda'z_2)\\
&\equiArrow e(\tr(\lambda z)) = 1\\
&\equiArrow \tr(\lambda z) \in \Z
\end{align*}
to see that $\psi_\lambda(z)$ vanishes if and only if $z$ lies in the components of $T^\infty(\a,m)$ belonging to $\lambda$ (cf.~representation~\eqref{Tinf-def} of $T^\infty(\a,m)$). Further, from $e(z)$ having a non-vanishing derivative it follows that all zeros of $\psi_\lambda(z)$ are simple.
Hence, the normal convergence of the product proves that $\Psi(\a,m,z)$ is a holomorphic function on $\H^2$ with simple roots at $T^\infty(\a,m)$.
To prove the $\Ginf$ invariance we make use of the decomposition
$\Ginfb \cong \a^{-1} \rtimes (\OK^\times)^2$
and show the invariance for both factors individually.
For $\ep^2 \in (\OK^\times)^2$ it is immediate by the definition of $\psi_\lambda(z)$ that we have
\[
\psi_\lambda(\ep^2 z)  = \psi_{\ep^2 \lambda}(z).
\]
Because $n$ is even we do not have to bother about the sign. Hence, the factors are only permuted by the action of $(\OK^\times)^2$.
However, for $\mu \in \a^{-1}$ we have
\begin{align*}
\psi_\lambda(z+\mu)
&= e(\lambda (z_1+\mu)) - e(-\lambda' (z_2+\mu'))\\
&= e(\lambda z_1)e(\lambda\mu) - e(-\lambda' z_2)e(-\lambda' \mu')\\
&= e(\lambda\mu) \br{e(\lambda z_1) - e(-\lambda' z_2)e(-\lambda\mu)e(-\lambda' \mu')}\\
&= e(\lambda\mu) \psi_{\lambda}(z).
\end{align*}
Here we used $\tr(\lambda\mu) \in \Z$ which is true because $\a\d^{-1}$ is the trace dual of $\a^{-1}$.
Analogously, we can factor $e(-\lambda'\mu')$ out to obtain
\[
\psi_\lambda(z+\mu) = e(-\lambda'\mu') \psi_{\lambda}(z).
\]
In particular, we have $e(\lambda\mu)= e(-\lambda'\mu')$ which can also be seen directly using $\tr(\lambda\mu) \in \Z$.
The set $\Lambda^+(\a,m)$ decomposes into finitely many $(\OK^\times)^2$ orbits. For each orbit we have
\[
\prod_{k \in \Z} \psi_{\ep_0^{2k} \lambda}(z+\mu)^n
= \prod_{k \in \Z} \psi_{\ep_0^{2k}\lambda}(z)^n \cdot \prod_{k \in \Z} e(\lambda \ep_0^{2k} \mu)^n.
\]
To compute the later product we use
\[
\prod_{k \in \Z} e(\lambda \ep_0^{2k} \mu)
= \prod_{k=1}^\infty e(\lambda \ep_0^{-2k} \mu) \cdot \prod_{k=0}^\infty e(-\lambda' \ep_0^{-2k} \mu').
\]
Using the functional equation, this boils down to computing the sum
\begin{align*}
\sum_{k=1}^\infty \lambda \ep_0^{-2k} \mu - \sum_{k=0}^\infty \lambda' \ep_0^{-2k} \mu'
&= \lambda \mu \frac{\ep_0^{-2}}{1-\ep_0^{-2}} - \lambda' \mu' \frac{1}{1-\ep_0^{-2}}\\
&= \lambda \mu \frac{1}{\ep_0^{2}-1} - \lambda' \mu' \br{ \frac{1}{1-\ep_0^{2}} }'
 = \tr\br{ \frac{\lambda \mu }{\ep_0^{2}-1} }.
\end{align*}
Hence, we have proven
\[
\prod_{k \in \Z} e(\lambda \ep_0^{2k} \mu)^n = e\br{ \tr\br{ \frac{n \lambda \mu }{\ep_0^{2}-1} }}.
\]
By the choice of $n$ we have
\[
\frac{n\lambda}{\ep_0^{2}-1} \in \a\d^{-1}
\]
which proves
\[
\tr\br{ \frac{n\lambda \mu }{\ep_0^{2}-1} } \in \Z.
\]
Hence, the infinite product
\[
\prod_{k \in \Z} \psi_{\ep_0^{2k}\lambda}(z)^n
\]
is invariant under translation by $\a^{-1}$ and therefore invariant under $\Ginf$. The same holds for $\Psi(\a,m,z)^n$ which is a finite product of such factors.
\end{proof}

An easy way to come up with an admissible sign function $\sigma$ is to partition the set $\Lambda^+(\a,m)$ into a lower and an upper part with respect to a fixed $w \in (\R^+)^2$ using the trace by
\[
\sigma_w : \Lambda^+(\a,m) \to \set{\pm 1},
\quad \sigma_w (\lambda) := \begin{cases}
+1,&\quad \tr(\lambda w) < 0,\\
-1,&\quad \tr(\lambda w) \ge 0.
\end{cases}
\]
The next proposition states a useful representation of $\Psi_{\sigma_w}$.

\begin{proposition} \label{borcherds-product-using-weyl} 
Let $w \in (\R^+)^2$ and let
\[
\Lambda_w :=
\set{ \lambda \in \Lambda^+(\a,m) \defC \tr(\lambda w) \ge 0}
\cup
\set{ \lambda \in \Lambda^-(\a,m) \defC \tr(\lambda w) > 0}.
\]
Then we have
\[
\Psi_{\sigma_w}(\a,m,z) = e\br{ \tr\br{ \rho(\a,m,w) z } } \prod_{\lambda \in \Lambda_w} \br{1- e(\tr(\lambda z))}.
\]
\end{proposition}
\begin{proof}
This proposition is proven by exploiting the functional equation of the exponential function and using $R(\a,m,w)$, the set of reduced $\lambda \in \Lambda^+(\a,m)$ with respect to $w$, to express the elements of $\Lambda_w$.
\end{proof}

The classic approach introducing the local Borcherds product makes use of Weyl chambers (cf.~\cite[p.~153, eq. (3.13)]{123mf}). The next corollary shows that the resulting product is the same.
\begin{corollary}
Let $W \in W(\a,m)$ be a Weyl chamber of index $m$. Let us fix one $z_0 \in W$ to define $\sigma(\lambda) :=-\sgn(\tr(\lambda y_0))$. Then we have
\[
\Psi_\sigma(\a,m,z) = e\br{ \tr\br{ \rho(\a,m,W) z } } \prod_{ \substack{\lambda \in \Lambda(\a,m) \\ (\lambda,W)>0}  } \br{1- e(\tr(\lambda z))}.
\]
\end{corollary}
\begin{proof}
Using $w:=y_0$, we have $\sigma=\sigma_{w}$, $\rho(\a,m,W) = \rho(\a,m,w)$ and
\[
\set{\lambda \in \Lambda(\a,m) \defC (\lambda,W)>0} = \Lambda_w
\]
with $\Lambda_w$ defined as in Proposition~\ref{borcherds-product-using-weyl}. Hence, the result is nothing but a direct application of Proposition~\ref{borcherds-product-using-weyl}.
\end{proof}

\begin{proposition} \label{BP-vanish-order} 
Let $(\alpha,\beta)$ be a totally positive basis of $\a^{-1}$ and $n \in \N$ with
\[
\frac{n}{1-\ep_0^2} \in \OK.
\]
Then $\Psi(\a,m,z)^n$ is invariant under $\a^{-1}$ and possesses a holomorphic extension to $u=0$ and $v=0$ in local coordinates $(u,v)$ with respect to $(\alpha,\beta)$. At $u=0$ ($v=0$, respectively) the product vanishes. Its order of vanishing along $u$ ($v$, respectively) is given by $n\tr(\rho(\a,m,\alpha)\alpha)$ ($n\tr(\rho(\a,m,\beta)\beta)$, respectively).
\end{proposition}
\begin{proof}
Since $\alpha$ and $\beta$ (and hence $u$ and $v$) are interchangeable, we prove the result for $v$ only. By Proposition~\ref{borcherds-product-using-weyl} the Borcherds product is expressible as
\[
e\br{ \tr\br{ \rho(\a,m,\beta) z } } \prod_{\lambda \in \Lambda_\beta} \br{1- e(\tr(\lambda z))}.
\]
By Lemma~\ref{exponentials-in-local} each factor of the product is $\a^{-1}$ invariant and we have
\[
\prod_{\lambda \in \Lambda_\beta} \br{1- e(\tr(\lambda z))}= 
\prod_{\lambda \in \Lambda_\beta} (1-u^{\tr(\lambda \alpha)}v^{\tr(\lambda \beta)})
\]
in local coordinates.
We list some facts we know about the exponents of $u$ and $v$:
\begin{enumerate}[(i)]
\item We have $\tr(\lambda \alpha) \in \Z$ and $\tr(\lambda \beta) \in \N_0$ for all $\lambda \in \Lambda_\beta$.
\item For each $m \in \Z$ there are at most two $\lambda \in \Lambda_\beta$ with $\tr(\lambda \alpha) = m$ ($\tr(\lambda \beta) = m$ respecively).
\item There are only finitely many $\lambda \in \Lambda_\beta$ with $\tr(\lambda \alpha)<0$.
\end{enumerate}
Those facts imply that the product converges normally to a holomorphic function in $u$ and $v$ in the domain
\[
\set{ (u,v) \in \C^2 \defC 0<|u|<1 , |v|<1 }
\]
and that it does not vanish at $v=0$.
Hence, we are left with inspecting the factor in front of the product $e(\tr(\rho z))$ (for simplicity we abbreviate $\rho:=\rho(\a,m,\beta)$ for the rest of the proof). This factor might not be $\a^{-1}$ invariant but the $n$-th power is because we have $e(\tr(\rho z))^n=e(\tr(n\rho z))$.
Now by assumption on $n$ and the definition of the Weyl vector $\rho$ we have $n\rho \in \a\d^{-1}$. Hence, Lemma~\ref{exponentials-in-local} again implies the $\a^{-1}$ invariance of $e(\tr(n\rho z))$ and
\[
e(\tr(n\rho z)) = u^{\tr(n\rho \alpha)}v^{\tr(n\rho \beta)}.
\]
It is easy to see that the Weyl vector $\rho$ is totally positive. That makes $\tr(\rho \beta)$ positive which finishes the proof.
\end{proof}

\begin{corollary} \label{Bp-log-sing} 
The function
\[
\log  \abs{\Psi(\a,m,z)}^2
\]
is well-defined in a neighborhood of the exceptional divisor $E^\infty(\a) \se \XaH$ and has logarithmic singularities along the divisor $T^\infty(\a,m)+Z^\infty(\a,m)$.
\end{corollary}
\begin{proof}
Let $n \in \N$ be like in Proposition~\ref{power-of-BP-is-invariant}. Then $\Psi(\a,m,z)^n$ is invariant under $\Ginf$. This shows that $\Psi(\a,m,z)^n$ is well-defined on a punctured neighborhood of $\infty$ in $\XaB$ and holomorphic there. With Proposition~\ref{BP-vanish-order} we obtain that $\Psi(\a,m,z)^n$ is well-defined on $E^\infty(\a)$ as well, hence on a neighborhood of $E^\infty(\a)$ in $\XaH$, and that this extension is holomorphic.
With $\Psi(\a,m,z)^n$ being well-defined, of course also
\[
\log  \abs{\Psi(\a,m,z)}^2 = \frac{1}{n} \log \abs{\Psi(\a,m,z)^n}^2
\]
is well-defined.
Now, we come to prove the stated logarithmic singularities. For this we have to show that the divisor of the holomorphic function $\Psi(\a,m,z)^n$ agrees with $n (T^\infty(\a,m)+Z^\infty(\a,m))$.
By Proposition~\ref{power-of-BP-is-invariant} the function $\Psi(\a,m,z)^n$ vanishes of order $n$ at $T^\infty(\a,m)$ in $\Xa$. By Proposition~\ref{BP-vanish-order} the divisor $n Z^\infty(\a,m)$ provides the correct multiplicities for the vanishing of $\Psi(\a,m,z)^n$ along $E^\infty(\a)$. To see that, recall definition~\eqref{Zinf-def} of $Z^\infty(\a,m)$ with $(\alpha,\beta):=(A_{k-1},A_k)$ to realize that the multiplicities of the components $S_k$ of $nZ^\infty(\a,m)$ are precisely defined to match the multiplicites of the zeros of $\Psi(\a,m,z)^n$ along $S_k$.
\end{proof}

\subsection{Growth analysis}
In this subsection we prove that the regularized automorphic Green functions $\Phi(\a,m,z)$ are actual Green functions, i.e., $\Phi(\a,m,z)$ is a pre-log-log Green function on $\XaH$ for the divisor $Z(\a,m)$. On $\Xa$ this is already clear because $\Phi(\a,m,z)$ has logarithmic singularties along $-T(\a,m)$ and is elsewhere smooth, even real analytic. On the cusps however, $\Phi(\a,m,z)$ is not smooth anymore, even after subtracting the logarithmic singularities and the log-log growth of $- q(\a,m)\log(16 \pi^2 y_1y_2)$.

We start with three lemmata which are straight forward to prove using Remark~\ref{pre-log-log-condition}.
\begin{lemma} \label{ddc-for-absolute-powers} 
Let $a_1,a_2, b_1,b_2 \in \Z$ and  $\alpha,\beta \in \R$ with
\[
\alpha + a_1 + a_2 > 0
\und
\beta + b_1 + b_2 > 0.
\]
Then the function
\[
f: (\C^\times)^2 \to \C,\quad
f(u,v) = u^{a_1}\ub^{a_2} |u|^\alpha \cdot v^{b_1}\vb^{b_2} |v|^\beta
\]
is a pre-log-log growth form along $uv=0$.
\end{lemma}

\begin{lemma} \label{ddc-for-log-one-minusabsolute-powers} 
Let $a_1,a_2, b_1,b_2 \in \Z$ and  $\alpha,\beta \in \R$ with
\[
\alpha + a_1 + a_2 > 0
\und
\beta + b_1 + b_2 > 0.
\]
Then the function
\[
f: (\C^\times)^2 \to \C,\quad
f(u,v) = \log \abs{1- u^{a_1}\ub^{a_2} |u|^\alpha \cdot v^{b_1}\vb^{b_2} |v|^\beta}^2
\]
is a pre-log-log growth form along $uv=0$.
\end{lemma}

\begin{lemma} \label{green-equation-for-log-t} 
Let $(\alpha,\beta)$ be a totally positive basis of $\a^{-1}$. The $\a$ invariant function
\[
f:\H^2 \to \C,\quad
f(z):=\log(y_1y_2)\]
expressed in local coordinates $(u,v)$ with respect to $(\alpha,\beta)$ is a pre-log-log growth form along $uv=0$.
\end{lemma}

\begin{theorem} \label{Phi-is-green} 
The function $\Phi(\a,m,z)$ is a pre-log-log Green function on $\XaH$ with respect to the divisor $Z(\a,m)$.
\end{theorem}
\begin{proof}
As already mentioned in the beginning of the subsection we do not have to care about $\Xa$ anymore.
Therefore, the focus of this proof lies on the cusps. Because of the isomorphism
$\overline{X(\a\b^2)} \isoArrow \XaH$
together with
\[ \Phi(\b,m,s,Mz) = \Phi(\a^2\b,m,s,z) \]
for $M \in \SLM{\a}{\b}$ by Proposition~\ref{first-Phi-prop} it is enough to consider the cusp~$\infty$.
We write
\[
\Phi(\a,m,z) = f_1(z) + f_2(z) + f_3(z) + f_4(z) + f_5(z) + f_6(z)
\]
near the cusp~$\infty$ as sum of six parts according to Theorem~\ref{nice-Phi-rep}.
We show that the functions $f_j$ with $j \in \set{1,2,3,4,5}$ are pre-log-log growth forms along $E^\infty(\a)$ and that $f_6$ has logarithmic singularities along the divisor $-(T^\infty(\a,m)+Z^\infty(\a,m))$. Note that the divisor $T^\infty(\a,m)+Z^\infty(\a,m)$ is the part of $Z(\a,m)$ in small neighborhoods of $E^\infty(\a)$.
For proving the pre-log-log growth we express $f_j$ in local coordinates $(u,v)$ with respect to a totally positive basis $(\alpha,\beta)$ of $\a^{-1}$.
Let us make our decomposition of $\Phi(\a,m,z)$ precise:
\begin{align*}
f_1(z) &:= L(\a,m),\\
f_2(z) &:= - q(\a,m)\log(16 \pi^2 y_1y_2),\\
f_3(z) &:= \sum_{\substack{\nu \in \a\d^{-1} \\ \nu \gg 0 }} \frac{2 \pi}{D} \sqrt{ \frac{mN(\a)}{|N(\nu)|} } \sum_{b=1}^\infty  \frac{G^b(\a,m,\nu)}{b} I_{1} \br{ \frac{4 \pi}{b} \sqrt{ \frac{m |N(\nu)|}{N(\a)D}}}\\
&\times \ \br{e(\tr(\nu z)) + \overline{e(\tr(\nu z))}},\\
f_4(z) &:= \sum_{\substack{\nu \in \a\d^{-1} \\ \nu >0, \, \nu'<0 }} \frac{2 \pi}{D} \sqrt{ \frac{mN(\a)}{|N(\nu)|} } \sum_{b=1}^\infty  \frac{G^b(\a,m,\nu)}{b} J_{1} \br{ \frac{4 \pi}{b} \sqrt{ \frac{m |N(\nu)|}{N(\a)D}}}\\
&\times \  \br{ e(\nu z_1)\overline{ e(-\nu'z_2) } + \overline{e(\nu z_1)} e(-\nu'z_2) },\\
f_5(z) &:= \log \prod_{\lambda \in \Lambda^+(\a,m)} \abs{ 1- e(|\lambda| z_1) \overline{ e(|\lambda'| z_2)}}^2
= \sum_{\lambda \in \Lambda^+(\a,m)} \log \abs{ 1- e(\lambda z_1) \overline{ e(-\lambda' z_2)}}^2,\\
f_6(z) &:=  -\log \prod_{\lambda \in \Lambda^+(\a,m)} \abs{ e(|\lambda| z_1) -e(|\lambda'| z_2) }^2 = -\log \abs{\Psi(\a,m,z)}^2.
\end{align*}
The function $f_1$ is constant, hence it is a pre-log-log growth form.
The function $f_2$ was considered (up to constants) in Lemma~\ref{green-equation-for-log-t}.
The function $f_3$ is real analytic even at $uv=0$ because of
\begin{align*}
e(\tr(\nu z)) = u^{\tr(\alpha \nu)} v^{\tr(\beta \nu)}
\und
\overline{e(\tr(\nu z))} = \overline{u}^{\tr(\alpha \nu)} \overline{v}^{\tr(\beta \nu)}
\end{align*}
by Lemma~\ref{exponentials-in-local}. Note that $\tr(\alpha \nu),\tr(\beta\nu) \in \N$. Hence, it is a pre-log-log growth form.
Unfortunately, the function $f_4$ is not even differentiable at $uv=0$ but at least continuous. We have by Lemma~\ref{exponentials-in-local}
\begin{align*}
e(\nu z_1)\overline{ e(-\nu'z_2) }
&= u^{\alpha \nu} \ub^{-\alpha'\nu'} v^{\beta \nu} \vb^{-\beta'\nu'}\\
&= u^{\tr(\alpha \nu)} |u|^{-2 \alpha'\nu'} v^{\tr(\beta \nu)} |v|^{-2 \beta'\nu'}
\end{align*}
and
\[
\overline{e(\nu z_1)} e(-\nu'z_2)
=\ub^{\tr(\alpha \nu)} |u|^{-2 \alpha'\nu'} \vb^{\tr(\beta \nu)} |v|^{-2 \beta'\nu'}.
\]
The advantage of having integer powers on $u$, $\ub$, $v$ and $\vb$ is that it is well-defined without specifying a branch of the logarithm.
Since $\nu>0$ and $\nu'<0$, we have
\[
\tr(\alpha \nu)-2 \alpha'\nu' = \alpha \nu - \alpha'\nu' > 0
\und
\tr(\beta \nu)-2 \beta'\nu' = \beta \nu - \beta
\nu'> 0.
\]
Hence, the claim for $f_4$ follows by Lemma~\ref{ddc-for-absolute-powers}.
Considering $f_5$, we see that we can write each summand in local coordinates using the same identity and get
\[
\log \abs{ 1- e(\lambda z_1) \overline{ e(-\lambda' z_2)}}^2 = 
\log \abs{ 1- u^{\tr(\alpha \lambda)} |u|^{-2 \alpha'\lambda'} v^{\tr(\beta \lambda)} |v|^{-2 \beta'\lambda'}  }^2.
\]
Because of $\lambda>0$ and $\lambda'<0$ we can apply Lemma~\ref{ddc-for-log-one-minusabsolute-powers} to achieve the claim for $f_5$.
Now, we are left with
\[
f_6(z) = -\log \abs{\Psi(\a,m,z)}^2
\]
for which we have proven the claim already in Corollary~\ref{Bp-log-sing}.
\end{proof}

\section{Smooth decomposition of automorphic Green functions} \label{decomp-section} 

\subsection{A valuable representation using the hypergeometric function} 

In this subsection we follow the idea (for example present in \cite{bruinier2021cm}) to express $\Phi(\a,m,s,z)$ using the Gaussian hypergeometric function $\hF(a,b;c;z)$.
This yields a valuable decomposition
\[
\Phi(\a,m,s,z) = \sum_{n=0}^\infty \Phi_n(\a,m,s,z)
\]
into smooth, $\Ga$ invariant functions $\Phi_n(\a,m,s,z)$. Using this decomposition, a lot of already known results about $\Phi(\a,m,s,z)$ and $\Phi(\a,m,z)$ can be reproven.
Some of those proofs reveal new perspectives on the old results.
For example, computing the Fourier expansions of the functions $\Phi_n(\a,m,s,z)$ yields new formulae for the Fourier coefficients of $\Phi(\a,m,s,z)$.
However, the motivation for the author to look at this decomposition was to prove the integrability of $\Phi(\a,m,z)$ and understand the growth behavior of
\[
\int_{\Xa} |\Phi(\a,m,z)| \omega^2
\]
for large $m$ which is essential for the main result of \cite{buckdiss}. Those two results can be found in Corollary~\ref{Phi-integrable} and Theorem~\ref{Phi-integrable-estimate}. The main work towards these theorems is done in the current section.

The main ingredient in Definition~\ref{Phi-amsz-def} of $\Phi(\a,m,s,z)$ is $Q_{s-1}(x)$, the Legendre function of the second kind. This however has the nice representation
\begin{align} \label{Qs-as-hyper} 
Q_{s-1}(x) = \frac{\Gamma(s)^2}{2\Gamma(2s)} \br{ \frac{2}{1+x} }^s \hF\br{s,s;2s;\frac{2}{1+x}}
\end{align}
using the hypergeometric function $\hF(a,b;c;z)$
which follows from \cite[14.3.7 and 15.8.13]{handbook} together with the Legendre duplication formula. The hypergeometric function again is defined by its power series
\begin{align} \label{hyper-power} 
\begin{split}
\hF(a,b;c;z)
:= \sum_{n=0}^\infty \frac{(a)_n (b)_n}{(c)_n} \frac{z^n}{n!}
= \frac{\Gamma(c)}{\Gamma(a) \,\Gamma(b)} \sum_{n=0}^\infty \frac{ \Gamma(a+n) \Gamma(b+n) }{ \Gamma(c+n)} \frac{z^n}{n!}
\end{split}
\end{align}
which implies
\[
Q_{s-1}(x) = \frac{1}{2} \sum_{n=0}^\infty \frac{\Gamma(s+n)^2}{\Gamma(2s+n)} \frac{1}{n!} \br{\frac{2}{1+x}}^{n+s}.
\]
Plugged into Definition~\ref{Phi-amsz-def} we get
\[
\Phi(\a,m,s,z) =  \frac{1}{2}  \sum_{n=0}^\infty \frac{\Gamma(s+n)^2}{\Gamma(2s+n)} \frac{1}{n!} \sum_{ \substack{ A \in L(\a)^\vee \\ \det(A) = m/(N(\a)D) }  }  \br{1+g(A,z)}^{-(n+s)}.
\]
Defining
\begin{align} \label{Psi-def} 
\Psi(\a,m,s,z) :=  \sum_{ \substack{ A \in L(\a)^\vee \\ \det(A) = m/(N(\a)D) }  }  \br{1+g(A,z)}^{-s},
\end{align}
we obtain
\[
\Phi(\a,m,s,z) =  \sum_{n=0}^\infty \Phi_n(\a,m,s,z)
\quad\text{with}\quad
\Phi_n(\a,m,s,z) := \frac{\Gamma(s+n)^2}{\Gamma(2s+n)} \frac{\Psi(\a,m,s+n,z)}{2n!}.
\]
The convergence of $\Psi(\a,m,s,z)$ for $\Re(s)>1$ follows directly from the convergence of $\Phi(\a,m,s,z)$. Here, $\Psi(\a,m,s,z)$ is even well-defined for $z \in T(\a,m)$ and smooth in $z$ since $(1+x)^{-s}$ has no singularity at $x=0$. Furthermore, $\Psi(\a,m,s,z)$ is holomorphic in $s$.

It follows that the functions $\Phi_n(\a,m,s,z)$ are holomorphic in $s$, $\Ga$ invariant and smooth in $z$ on $\H^2$ for $\Re(s)>1-n$. Inductively, one can show that for all $N \in \N_0$
\[
\sum_{n=N}^\infty \Phi_n(\a,m,s,z)
\]
converges for $\Re(s)>1-N$ to a $\Ga$ invariant and smooth function on $\H^2 \setminus T(\a,m)$ which is holomorphic in $s$ (in particular $N=1$ implies convergence for $\Re(s)>0$). By Theorem~\ref{Phi-cont-theorem} we know that $\Phi(\a,m,s,z)$ has a meromorphic extension to $\Re(s)>3/4$ for $z \in \H^2 \setminus T(\a,m)$ with simple pole at $s=1$ of residue $q(\a,m)$. It follows that $\Phi_0(\a,m,s,z)$ has a meromorphic extension to $\Re(s)>3/4$ with simple pole at $s=1$ of residue $q(\a,m)$.
We define
\[
\Phi_0(\a,m,z) := \mathcal{C}_{s=1} \sq{\Phi_0(\a,m,s,z)}
\]
and get
\[
\Phi(\a,m,z) = \Phi_0(\a,m,z) + \sum_{n=1}^\infty \Phi_n(\a,m,1,z).
\]

\subsection{Fourier expansion of the decomposition} \label{fc-decomp} 
We proceed analogously to Subsection~\ref{fourier-exp-and-reg} and write
\[
\Psi(\a,m,s,z)
= \Psi^0(\a,m,s,z) + 2\sum_{b=1}^\infty \Psi^b(\a,m,s,z)
\]
with
\[
\Psi^b(\a,m,s,z) :=\sum_{ \substack{ A =\sMatrix{a}{\lambda'}{\lambda}{b} \in L(\a)^\vee \\ \det(A) = m/(N(\a)D) }  }  \br{1+g(A,z)}^{-s}.
\]
The functions $\Psi^b(\a,m,s,z)$ are invariant under $\Ginf$ as $\Phi^b(\a,m,s,z)$ in Subsection~\ref{fourier-exp-and-reg}. Hence, they are $\a^{-1}$ periodic and possess a Fourier expansion. Again, we treat the cases $b=0$ and $b \in \N$ separately and start with $b \in \N$.
We have with $B :=m/(N(\a)Db^2)$ and $R^b(\a,m)$ defined as in Subsection~\ref{fourier-exp-and-reg}
\begin{align*}
\Psi^b(\a,m,s,z)
&= \sum_{ \substack{a \in \Z/N(\a),\:\lambda \in \a\d^{-1}/N(\a) \\ ab-N(\lambda) = m/(N(\a)D) } }  \br{1+ \frac{|bz_1z_2-\lambda z_1-\lambda'z_2+a|^2}{4y_1y_2 m/(N(\a)D)} }^{-s}\\
&= \sum_{ \substack{a \in \Z/N(\a),\:\lambda \in \a\d^{-1}/N(\a) \\ ab-N(\lambda) = m/(N(\a)D) } }  \br{1+ \frac{| (z_1-\lambda'/b)(z_2-\lambda/b) +B |^2}{4y_1y_2 B} }^{-s}\\
=\sum_{\lambda \in R^b(\a,m)} &\sum_{\mu \in \a^{-1}}
\br{1+ \frac{ \abs{ \br {z_1 + \mu +  \frac{\lambda'}{N(\a)b}} \br {z_2 + \mu' +  \frac{\lambda}{N(\a)b}} + B }^2}{4y_1y_2 B} }^{-s}.
\end{align*}
Hence, the problem is reduced to computing the Fourier expansion of the $\a^{-1}$ periodic function $\tilde H_s^B (\a^{-1},z)$ with
\begin{align*}
\tilde H_s^B (\b,z) := \sum_{\mu \in \b} \br{ 1 + \frac{|(z_1+\mu)(z_2+\mu')+B|^2}{4y_1y_2B} }^{-s}.
\end{align*}
Namely, let
\[
\tilde H_s^B (\b,z) = \sum_{\nu \in (\b\d)^{-1} } \tilde b_s^B(\b,\nu,y) e(\tr(\nu x))
\]
be the Fourier expansion of $\tilde H_s^B (\b,z)$. Then we have
\begin{align*}
\Psi^b(\a,m,s,z)
&= \sum_{\nu \in \a\d^{-1}} G^b(\a,m,\nu)  \tilde b_s^B(\a^{-1},\nu,y) e(\tr(\nu x))
\end{align*}
with $G^b(\a,m,\nu)$ defined as in equation~\eqref{gaus-sum}.
By Poisson summation the Fourier coefficients are then given by
\begin{align} \label{tilde-b-formula} 
\tilde b_s^B(\b,\nu,y) = \frac{1}{\vol(\b)} \int_{\R^2}  \br{ 1 + \frac{|z_1z_2+B|^2}{4y_1y_2B} }^{-s}  e(-\tr(\nu x)) dx_1 dx_2.
\end{align}
For $\nu \ne 0$, the double integral is too complicated to be solved explicitly. Only one of the integrals can be solved explicitly, for the second one the author did not come up with an explicit solution.
However, for our purpose it is enough to estimate $|\tilde b_1^B(\b,\nu,y)|$.
Nevertheless, for $\nu = 0$ an estimate of $\tilde b_1^B(\b,0,y)$ is not enough because the series
\begin{align} \label{meromorphic-tilde-b-series}
\sum_{b=1}^\infty G^b(\a,m,0)  \tilde b_s^{m/(N(\a)Db^2)}(\a^{-1},0,y) 
\end{align}
diverges at $s=1$. Rather, we have to determine $\tilde b_s^B(\b,0,y)$ explicitly to compute the meromorphic continuation at $s=1$ of \eqref{meromorphic-tilde-b-series} and extract (or estimate) the constant term.

\begin{lemma} \label{fc-b-tilde-upper-bound} 
Let $B>0$, $\b \in \IK$ and $\nu \in (\b\d)^{-1}$. Then we have (cf.~equation~\eqref{alpha-beta-def} for the definition of $\alpha(\cdot,\cdot)$)
\[
\abs {\tilde b_1^B(\b,\nu,y)} \le \frac{4 B \pi^2}{\vol(\b)} \exp(- 2\pi \alpha(\nu y_1,\nu'y_2)).
\]
\end{lemma}
\begin{proof}
We have to estimate the integral given by equation~\eqref{tilde-b-formula} at $s=1$:
\begin{align*}
\tilde b_1^B(\b,\nu,y) = 
& \frac{1}{\vol(\b)}  \int_{\R^2}  \br{ 1 + \frac{|z_1z_2+B|^2}{4y_1y_2B} }^{-1}  e(-\tr(\nu x)) dx_1 dx_2\\
= & \frac{4y_1y_2B}{\vol(\b)}  \int_{\R^2} \br{ 4y_1y_2B +|z_1z_2+B|^2 }^{-1}  e(-\tr(\nu x)) dx_1 dx_2.
\end{align*}
Now, using the identity
\[
4y_1y_2 B +|z_1z_2+B|^2
= |z_2|^2 \br{ \br{ x_1 + \frac{Bx_2}{|z_2|^2} }^2 + \br{y_1 + \frac{By_2}{|z_2|^2}}^2 }
\]
the double integral is given by
\begin{align*}
\int_{\R} |z_2|^{-2} \int_{\R}  \br{ \br{ x_1 + \frac{Bx_2}{|z_2|^2} }^2 + \br{y_1 + \frac{By_2}{|z_2|^2}}^2 }^{-1}  e(-\tr(\nu x)) dx_1 dx_2\\
= \int_{\R} |z_2|^{-2} \br { \int_{\R}  \br{ x_1^2 + a(y_1,z_2)^2 }^{-1}  e(-\nu x_1) dx_1 } e\br{\frac{\nu Bx_2}{|z_2|^2} -\nu x_2 }  dx_2
\end{align*}
with $a(y_1,z_2) =: y_1 + \frac{By_2}{|z_2|^2}$. Using \cite[p.~8, eq. (11)]{1954tables} (which holds for $\nu=0$ as well, even though this case is omitted in the reference), we get for the inner integral
\begin{align*}
\int_{\R}  \br{ x_1^2 + a(y_1,z_2)^2 }^{-1}  e(-\nu x_1) dx_1
=\ &2\int_{0}^\infty \br{ x_1^2 + a(y_1,z_2)^2 }^{-1}  \cos(2 \pi |\nu| x_1) dx_1\\
=\ &\pi \frac{\exp(- 2 \pi |\nu| a(y_1,z_2) )}{a(y_1,z_2)}.
\end{align*}
Coming back to our double integral, we estimate
\begin{align*}
&\abs {\int_{\R} |z_2|^{-2} \br{ \pi \frac{\exp(- 2 \pi |\nu| a(y_1,z_2) )}{a(y_1,z_2)} } e\br{\frac{\nu Bx_2}{|z_2|^2} -\nu x_2 }  dx_2}\\
\le\ &\pi \int_{\R} \frac{\exp(- 2 \pi |\nu| a(y_1,z_2) )}{a(y_1,z_2)|z_2|^2} dx_2\\
\le\ &\frac{\pi \exp(-2 \pi |\nu| y_1)}{y_1} \int_{\R} \frac{1}{x_2^2+y_2^2} dx_2
= \frac{\pi^2 \exp(-2 \pi |\nu| y_1)}{y_1y_2} .
\end{align*}
Hence, in total we have shown
\[
\abs {\tilde b_1^B(\b,\nu,y)} \le  \frac{4  B \pi^2}{\vol(\b)}  \exp(-2 \pi |\nu| y_1).
\]
For symmetry reasons we have
\[
\abs {\tilde b_1^B(\b,\nu,y)} \le  \frac{4  B \pi^2}{\vol(\b)}  \exp(-2 \pi |\nu'| y_2)
\]
as well which proves the claim.
\end{proof}

In order to compute $\tilde b_s^B(\b,0,y)$ explicitly, we use the following two lemmata.

\begin{lemma} \label{square-sum-to-s-integral} 
Let $a>0$ and $s \in \C$ with $\Re(s)>1/2$. Then we have
\[
\int_{\R} (x^2+a^2)^{-s} dx = a^{1-2s} \Beta(\tfrac{1}{2},s-\tfrac{1}{2}).
\]
Here, by $\Beta(x,y)$ we denote the beta function $\Beta(x,y) := \Gamma(x)\,\Gamma(y)/\Gamma(x+y)$.
\end{lemma}
\begin{proof}
The identity is proven by appropriate substitutions and the use of the integral representation (cf.~\cite[5.12.3]{handbook})
\begin{align} \label{beta-integral2} 
\Beta(x,y) = \int_0^\infty \frac{t^{x-1}}{(1+t)^{x+y}} dt
\end{align}
which holds for $\Re(x),\Re(y)>0$.
\end{proof}

\begin{lemma} \label{crazy-mathematica-integral} 
Let $\Re(s)>1/2$ and $b>0$. Then we have
\begin{align*}
\int_{\R} \frac{(x^2+b^2)^{1-2s}}{(x^2+1)^{1-s}} dx
= \Beta \br{\tfrac{1}{2},s-\tfrac{1}{2}} \hF \br{2s-1,s-\tfrac{1}{2};s;1-b^2}.
\end{align*}
\end{lemma}
\begin{proof}
Can be checked using computer algebra systems.
\end{proof}

\begin{lemma} \label{fc-b-tilde-0} 
For $B>0$, $\b \in \IK$ and $\Re(s)>1/2$ the constant Fourier coefficient of $\tilde H_s^B (\b,z)$ is given by
\[
\tilde b_s^B(\b,0,y) =  \frac{(4B)^s (y_1y_2)^{1-s} \Beta(\tfrac{1}{2},s-\tfrac{1}{2})^2}{\vol(\b)}  \hF (2s-1,s-\tfrac{1}{2};s;-B/(y_1y_2)).
\]
\end{lemma}
\begin{proof}
We can copy the proof of Lemma~\ref{fc-b-tilde-upper-bound} until the point of the substitution in the inner integral of the double integral. By that we get
\[
\tilde b_s^B(\b,0,y) = \frac{(4y_1y_2B)^s}{\vol(\b)} \int_{\R} |z_2|^{-2s}  \int_\R \br{ x_1^2 + a(y_1,z_2)^2 }^{-s}  dx_1   dx_2.
\]
Now using Lemma~\ref{square-sum-to-s-integral}, the inner integral computes to
\begin{align*}
a(y_1,z_2)^{1-2s} \Beta(\tfrac{1}{2},s-\tfrac{1}{2}).
\end{align*}
The integrand of the outer integral is then, up to the beta function factor, given by
\begin{align*}
|z_2|^{-2s} a(y_1,z_2)^{1-2s}
=y_1^{1-2s} \frac{ ( x_2^2+y_2^2 + By_2/y_1 )^{1-2s}  }{ (x_2^2+y_2^2)^{1-s} }.
\end{align*}
It follows
\begin{align*}
\int_\R &|z_2|^{-2s} a(y_1,z_2)^{1-2s} dx_2
 = (y_1y_2)^{1-2s} \int_\R \frac{ ( x^2 + b(y)^2 )^{1-2s}  }{ (x^2+1)^{1-s} } dx
\end{align*}
with $b(y)^2=1+B/(y_1y_2)$. The last integral is given using Lemma~\ref{crazy-mathematica-integral} by
\[
\Beta(\tfrac{1}{2},s-\tfrac{1}{2}) \hF (2s-1,s-\tfrac{1}{2};s;-B/(y_1y_2)).
\]
Collecting the omitted prefactors, we get the stated result.
\end{proof}

Now we come to the case $b=0$. Hence, we determine the Fourier expansion of $\Psi^0(\a,m,s,z)$.
We have
\begin{align*}
\Psi^0(\a,m,s,z)
&= \sum_{ \substack{a \in \Z/N(\a),\:\lambda \in \a\d^{-1}/N(\a) \\ -N(\lambda) = m/(N(\a)D) } }  \br{1+ \frac{|-\lambda z_1-\lambda'z_2+a|^2}{4y_1y_2 m/(N(\a)D)} }^{-s}\\
&= 2 \sum_{ \lambda \in \Lambda^+(\a,m) } \sum_{a \in \Z}  \br{1+ \frac{|\lambda z_1+\lambda'z_2+a|^2}{4y_1y_2 mN(\a)/D} }^{-s}.
\end{align*}

\begin{lemma} \label{fc-b0-tilde}
The series
\begin{align*}
\Psi^0(\a,m,s,z) =
2 \sum_{ \lambda \in \Lambda^+(\a,m) } \sum_{a \in \Z}  \br{1+ \frac{|\lambda z_1+\lambda'z_2+a|^2}{4y_1y_2 mN(\a)/D} }^{-s}
\end{align*}
converges normally for $z \in \H^2$ and $\Re(s)>1/2$ and has the Fourier expansion
\begin{align*}
\Psi^0(\a,m,s,z) = 2 \br{\frac{4y_1y_2 mN(\a)}{D}}^s  \Beta(\tfrac{1}{2},s-\tfrac{1}{2}) \sum_{ \lambda \in \Lambda^+(\a,m) }  (\lambda y_1-\lambda' y_2)^{1-2s}\\
+ \frac{4\pi^s}{\Gamma(s)} \br{\frac{4y_1y_2 mN(\a)}{D}}^s  \sum_{ \lambda \in \Lambda^+(\a,m) } \sum_{n=1}^\infty  \br{\frac{n}{\lambda y_1-\lambda' y_2}}^{s-1/2}\\
\times K_{s-1/2}(2 \pi n (\lambda y_1-\lambda' y_2))  \br{e(n\tr(\lambda x))+e(-n\tr(\lambda x))}.
\end{align*}
\end{lemma}
\begin{proof}
As in Lemma~\ref{Phi-0-fc}, we investigate the series over $a$ for each ${\lambda \in \Lambda^+(\a,m)}$ individually:
\begin{align*}
&\sum_{a \in \Z}  \br{1+ \frac{|\lambda z_1+\lambda'z_2+a|^2}{4y_1y_2 mN(\a)/D} }^{-s}\\
= &\sum_{a \in \Z}  \br{1+ \frac{ (\tr(\lambda x)+a)^2 + \tr(\lambda y)^2}{-4y_1y_2 \lambda \lambda'} }^{-s}\\
=  &\br{\frac{4y_1y_2 mN(\a)}{D}}^s \sum_{a \in \Z} \br{(\tr(\lambda x)+a)^2 +  (\lambda y_1-\lambda'y_2)^2 }^{-s}.
\end{align*}
Hence, we are interested in the Fourier expansion of the $\Z$ periodic function
\[
h_\gamma(s,x) := \sum_{a \in \Z} \br{(x+a)^2+\gamma^2}^{-s}
\]
with $\gamma>0$ (note that actually always $\gamma:=|\lambda y_1- \lambda'y_2|>0$ since $\lambda y_1 \ne \lambda'y_2$ due to $N(\lambda)<0$). 
It is straightforward to see that $h_\gamma(s,x)$ converges if and only if $\Re(s)>1/2$.
We have
\[
h_\gamma(s,x) = \sum_{n \in \Z} a_\gamma(s,n) e(nx)
\]
with
\[
a_\gamma(s,n) = \int_\R (x^2+\gamma^2)^{-s} e(-nx)dx.
\]
Lemma~\ref{square-sum-to-s-integral} yields
\[
a_\gamma(s,0) = \Beta(\tfrac{1}{2},s-\tfrac{1}{2}) \gamma^{1-2s} .
\]
For $n \ne 0$ we use \cite[p.~11, eq. (7)]{1954tables} (valid for $\Re(s)>0$)
\begin{align*}
a_\gamma(s,n)
&= 2 \int_0^\infty (x^2+\gamma^2)^{-s} \cos(2 \pi |n| x)dx\\
&= 2 \br{\frac{\pi|n|}{\gamma}}^{s-1/2} \sqrt{\pi} \Gamma(s)^{-1} K_{s-1/2}(2 \pi \alpha |n|)\\
&=   \frac{2 \pi^s}{\Gamma(s)} \br{\frac{|n|}{\gamma}}^{s-1/2} K_{s-1/2}(2 \pi \gamma |n|).
\end{align*}
We obtain
\begin{align*}
&\sum_{a \in \Z} \br{1+ \frac{|\lambda z_1+\lambda'z_2+a|^2}{4y_1y_2 mN(\a)/D} }^{-s}
= \br{\frac{4y_1y_2 mN(\a)}{D}}^s h_{|\lambda y_1-\lambda' y_2|}(s,\tr(\lambda x))\\
= &\br{\frac{4y_1y_2 mN(\a)}{D}}^s  \Beta(\tfrac{1}{2},s-\tfrac{1}{2}) |\lambda y_1-\lambda' y_2|^{1-2s}\\
 + &\frac{2\pi^s}{\Gamma(s)} \br{\frac{4y_1y_2 mN(\a)}{D}}^s  \ssum_{n \in \Z}  \abs{\frac{n}{\lambda y_1-\lambda' y_2}}^{s-1/2} K_{s-1/2}(2 \pi |n| |\lambda y_1-\lambda' y_2|)  e(\tr(\lambda n x))
\end{align*}
which proves the lemma. Here, the tick at the sum indicates that we do not sum over $n=0$.
\end{proof}
\begin{corollary} \label{Psi-0-fc} 
The function $\Psi^0(\a,m,1,z)$ has the Fourier expansion
\begin{align*}
 \frac{8 \pi y_1y_2 mN(\a)}{D} \sum_{ \lambda \in \Lambda^+(\a,m) } \sum_{n \in \Z} \frac{\exp(-2 \pi |n| (\lambda y_1-\lambda'y_2))}{\lambda y_1-\lambda'y_2} e(n\tr(\lambda x)).
\end{align*}
\end{corollary}
\begin{proof}
Plugging in $s=1$ into Lemma~\ref{fc-b0-tilde} with the identity $K_{1/2}(x)=\sqrt{\frac{\pi}{2x}} \exp(-x)$
yields the Fourier expansion.
\end{proof}

\section{Integrability and integrals} \label{integral-section}

In this section we compute the integral of $\Phi(\a,m,z)$. In order to do so, we compute the integrals of $\Psi(\a,m,s,z)$ and $\Phi(\a,m,s,z)$ as well. Our method of computing $\Phi(\a,m,z)$ requires the integrability of $\Phi(\a,m,z)$ first which is much more demanding to show than computing the actual integral afterwards. In the process we prove that the growth of
\[
\int_{\Xa} |\Phi(\a,m,z)| \omega^2
\]
is polynomial in $m$ which is an important ingredient for the main theorem of \cite{buckdiss}.

\begin{lemma} \label{Psi-core-integral} 
Let $\Re(s)>1$. Then we have
\[
\int_\H \br{1+ \frac{|z-i|^2}{4y}}^{-s} \frac{dxdy}{y^2} = \frac{4 \pi }{s-1}.
\]
\end{lemma}
\begin{proof}
We have by Lemma~\ref{square-sum-to-s-integral} 
\begin{align*}
\int_\H \br{1+ \frac{|z-i|^2}{4y}}^{-s} \frac{dxdy}{y^2}
&= 4^s \int_\H \br{4y+ (x^2+(y-1)^2)}^{-s} \frac{dxdy}{y^{2-s}}\\
&= 4^s \int_{0}^\infty \int_{-\infty}^\infty \br{x^2+(y+1)^2}^{-s} dx \ y^{s-2} dy\\
&= 4^s \int_{0}^\infty B(\tfrac{1}{2},s-\tfrac{1}{2})  (y+1)^{1-2s} y^{s-2} dy\\
&= 4^s  B(\tfrac{1}{2},s-\tfrac{1}{2}) \int_{0}^\infty \frac{ y^{s-2}}{(y+1)^{2s-1}}  dy\\
&= 4^s  B(\tfrac{1}{2},s-\tfrac{1}{2}) B(s-1,s).
\end{align*}
The last identity is due to equation~\eqref{beta-integral2} where we need $\Re(s)>1$.
By making use of the Legendre duplication formula and the functional equation of the gamma function one shows
\begin{align*}
4^s  B(\tfrac{1}{2},s-\tfrac{1}{2}) B(s-1,s)
&= \frac{8 \pi}{2s-2} = \frac{4 \pi}{s-1}.
\end{align*}
\end{proof}

In the next theorem we compute the integral of $\Psi(\a,m,s,z)$ over $\Xa$ by unfolding the integral. This technique allows us to compute the integral without determining a fundamental domain for $\Xa$.
\begin{theorem} \label{Psi-integral} 
For $\Re(s)>1$ we have
\[
\int_{\Xa} \Psi(\a,m,s,z) \omega^2 = \frac{4}{s-1} \vol(T(\a,m)).
\]
\end{theorem}
\begin{proof}
In this proof we will freely interchange integration and summation. 
Looking at the definition
\[
\Psi(\a,m,s,z) :=  \sum_{ \substack{ A \in L(\a)^\vee \\ \det(A) = m/(N(\a)D) }  }  \br{1+g(A,z)}^{-s},
\]
we see that for $s \in \R$ this is justified by Tonelli's theorem because all summands are positive. For $s \in \C$ we see
\[\abs{\Psi(\a,m,s,z)} \le \Psi(\a,m,\Re(s),z)\]
using the triangle inequality, hence Fubini's theorem can be applied by Lebesgue's dominated convergence theorem.

We start by some rewriting of $\Psi(\a,m,s,z)$. Subsequently, we explain every step.
\begin{align*}
\Psi(\a,m,s,z) 
&\overset{(i)}{=} 2\sum_{ \substack{A \in L(\a)^\vee / \setpm \\ \det(A) = m/(N(\a)D) }  }  \br{1+g(A,z)}^{-s}\\
&\overset{(ii)}{=} 2\sum_{ \substack{A \in \Ga \bs L(\a)^\vee / \setpm \\ \det(A) = m/(N(\a)D) }  } \sum_{M \in  \Ga / \Gas{ \pm A} }  \br{1+g(M.A,z)}^{-s}\\
&\overset{(iii)}{=} 2\sum_{ \substack{A \in \Ga \bs L(\a)^\vee / \setpm \\ \det(A) = m/(N(\a)D) }  } \sum_{M \in  \Ga / \Gas{ \pm A} }  \br{1+g(A,M^{-1} z)}^{-s}\\
&\overset{(iv)}{=} 2\sum_{ \substack{A \in \Ga \bs L(\a)^\vee / \setpm \\ \det(A) = m/(N(\a)D) }  } \sum_{M \in  \Gas{ \pm A}  \bs \Ga }  \br{1+g(A,M z)}^{-s}.
\end{align*}
In step~(i) we use the sign invariance of $g(A,z)$.
In step~(ii) we group up the summands by factoring out the action of $\Ga$ on $L(\a)^\vee / \setpm$. The resulting quotient is finite and each element corresponds to one component of $T(\a,m)$ viewed as divisor of $\Xa$. Now for each element in the quotient we have to sum over the whole $\Ga$ orbit to obtain all original summands back. This is what the inner sum does. We have to factor out the stabilizer
$\Gas{\pm A} := \set{ M \defC M \in \Ga \und M.A \in \set{\pm A} }$
in order to obtain every element in the orbit once.
In step~(iii) we use the invariance of $g(A,z)$ (cf.~equation \eqref{hg-transform}).
In step~(iv) we invert $\Ga / \Gas{ \pm A}$ which turns the left cosets into right cosets. This is compensated by inverting $M^{-1}$ as well.

Because the inner sum is invariant under $\Ga$ for each fixed $A \in L(\a)^\vee$ with $\det(A) = m/(N(\a)D)$, we can compute the integral of that inner sum over $\Xa = \Ga \bs \H^2$ first on its own. Again, we explain the equations step by step after the computation.
\begin{align*}
&\int_{\Ga \bs \H^2} \sum_{M \in  \Gas{\pm A}  \bs \Ga }  \br{1+g(A,M z)}^{-s} \omega^2\\
\overset{(v)}{=}& \int_{\Gas{\pm A} \bs \H^2}  \br{1+g(A,z)}^{-s} \omega^2\\
\overset{(vi)}{=}&\  2\int_{z_2 \in \Gas{ \pm A}' \bs \H}\int_{z_1 \in \H}  \br{1+g(A,z)}^{-s} \eta_1\eta_2\\
\overset{(vii)}{=}&\  2\int_{z_2 \in \Gas{ \pm A}' \bs \H}\int_{z_1 \in \H}  \br{1+\frac{d(z_1,ASz_2)}{4} }^{-s} \eta_1\eta_2\\
\overset{(viii)}{=}&\  2\int_{z_2 \in \Gas{ \pm A}' \bs \H}\int_{z_1 \in \H}  \br{1+\frac{d(z_1,i)}{4} }^{-s} \eta_1\eta_2\\
\overset{(ix)}{=}&\  \frac{2}{s-1}\int_{z_2 \in \Gas{ \pm A}' \bs \H}\eta_2
\overset{(x)}{=} \frac{2}{s-1} \vol(T_A).
\end{align*}
In step~(v) the actual unfolding takes place. Instead of integrating a sum of $\Gas{\pm A} \bs \Ga$ shifted functions over $\Ga \bs \H^2$ it is possible to integrate over $\Gas{\pm A} \bs \H^2$ in the first place and skip the sum and shifting.
In step~(vi) we use the fact that up to a set of measure zero a fundamental domain of $\Gas{\pm A} \bs \H^2$ is given by $\H \times \Gas{\pm A}' \bs \H$. Further, we use that $\omega^2 = 2 \eta_1\eta_2$ (cf.~equation~\eqref{omega-def}).
Step~(vii) is an application of Remark~\ref{gA-with-d-remark}.
Step~(viii) is a consequence of the $\GLRp$ invariance of $\eta_1$ and the hyperbolic distance. Since we integrate over all of $\H$ in the first argument, the reference point in the second argument is arbitrary.
Step~(ix) is an application of Lemma~\ref{Psi-core-integral} (note the scaling of $\eta$ with $(4\pi)^{-1}$ in equation~\eqref{omega-def}).
For step~(x) see equation~\eqref{Tm-volume-sum}.

In total we have
\begin{align*}
\int_{\Xa} \Psi(\a,m,s,z) \omega^2
=  2\sum_{ \substack{A \in \Ga \bs L(\a)^\vee / \setpm \\ \det(A) = m/(N(\a)D) }  }   \frac{2}{s-1} \vol(T_A).
=\frac{4}{s-1} \vol(T(\a,m))
\end{align*}
\end{proof}

This allows us to compute the integral of $\Phi(\a,m,s,z)$.

\begin{theorem} \label{Phi-s-integral} 
For $\Re(s)>1$ we have
\[
\int_{\Xa} \Phi(\a,m,s,z) \omega^2 = \frac{2 \vol(T(\a,m))}{s(s-1)}.
\]
\end{theorem}
\begin{proof}
To compute the integral we use the decomposition
\[
\Phi(\a,m,s,z) = \sum_{n=0}^\infty \frac{\Gamma(s+n)^2}{\Gamma(2s+n)} \frac{\Psi(\a,m,s+n,z)}{2n!}
\]
and Theorem~\ref{Psi-integral}.
We get
\begin{align*}
\int_{\Xa} \Phi(\a,m,s,z)
&= \sum_{n=0}^\infty \frac{\Gamma(s+n)^2}{\Gamma(2s+n)} \frac{\frac{4}{s+n-1} \vol(T(\a,m))}{2n!}\\
&= 2 \vol(T(\a,m)) \sum_{n=0}^\infty \frac{\Gamma(s+n)^2}{\Gamma(2s+n)} \frac{1}{n!}  \frac{1}{s+n-1}.
\end{align*}
Using the functional equation of the gamma function, the power series expansion of the hypergeometric function~\eqref{hyper-power} and \cite[15.4.2]{handbook} one shows the identity
\[
\sum_{n=0}^\infty \frac{\Gamma(s+n)^2}{\Gamma(2s+n)} \frac{1}{n!}  \frac{1}{s+n-1} =  \frac{1}{s(s-1)}
\]
which finishes the proof.
\end{proof}

\begin{proposition} \label{integral-Phi-0} 
The function $\Phi_0(\a,m,z)$ is integrable and we have
\[
\int_{X(\a)} \abs{\Phi_0(\a,m,z)} \omega^2 = O(m^2 \log(m))
\]
for large $m$.
\end{proposition}
\begin{proof}
The proof of this proposition relies on the results of Subsection~\ref{fc-decomp}. Because of its length and many tedious estimates it is skipped here. All the details can be found in \cite{buckdiss}. The rough idea is to replace each term in the Fourier series by its absolute value and estimate the integral over the resulting series. However, when dealing with the constant Fourier coefficient where the regularization takes place one has to work more subtle.
\end{proof}

\begin{corollary} \label{Phi-integrable}
The function $\Phi(\a,m,z)$ is integrable.
\end{corollary}
\begin{proof}
We make use of the decomposition
\begin{align} \label{Phi-decomp} 
\Phi(\a,m,z) = \Phi_0(\a,m,z) + \sum_{n=1}^\infty \Phi_n(\a,m,1,z).
\end{align}
The integrability of $\Phi_0(\a,m,z)$ is the statement of Proposition~\ref{integral-Phi-0}.
For the remaining series one argues similar as in the proof of Theorem~\ref{Phi-s-integral} and obtains integrability holomorphicity of the integral for $\Re(s)>0$.
\end{proof}

\begin{theorem} \label{Phi-integral} 
We have
\[
\int_{\Xa} \Phi(\a,m,z) \omega^2 = -2 \vol(T(\a,m)) = -q(\a,m)\zeta_K(-1),
\text{ thus }
q(\a,m) = \frac{2\vol(T(\a,m))}{\zeta_K(-1)}.
\]
In particular, for odd $D$ we obtain
\[
\vol(T(\a,m)) = \frac{\sigma(\a,m,-1)}{24}.
\]
\end{theorem}
\begin{proof}
By Corollary~\ref{Phi-integrable} we know that $\Phi(\a,m,z)$ is integrable. This allows us to apply Lebesgue's dominated convergence theorem
\begin{align*}
\int_{\Xa} \Phi(\a,m,z) \omega^2
&= \lim_{s \to 1} \br{ \int_{\Xa} \Phi(\a,m,s,z) \omega^2 - \int_{\Xa} \frac{q(\a,m)}{s-1} \omega^2 }\\
&= \lim_{s \to 1} \br{ \frac{2 \vol(T(\a,m))}{s(s-1)} - \frac{q(\a,m)}{s-1}\zeta_K(-1) }.
\end{align*}
Since the integral is finite by Corollary~\ref{Phi-integrable} the only possibility is
\[
2 \vol(T(\a,m)) = q(\a,m)\zeta_K(-1)
\]
which proves the stated identity about $q(\a,m)$.
The integral identity follows then with L'Hôpital's rule.
For odd $D$ we may use equation~\eqref{q-sigma} and $\zeta(-1)=-1/12$ to obtain
\[
\vol(T(\a,m))
= q(\a,m) \frac{\zeta_K(-1)}{2}
= -\frac{\sigma(\a,m,-1)}{L(-1,\chi_D)} \frac{\zeta(-1) L(-1,\chi_D)}{2}
= \frac{\sigma(\a,m,-1)}{24}.
\]
\end{proof}

\begin{corollary} \label{Phi-0-integral} 
We have
\[
\int_{\Xa} \Phi_0(\a,m,z) \omega^2 = -4 \vol(T(\a,m))
\]
and
\[
\int_{\Xa} \sum_{n=1}^\infty \Phi_n(\a,m,1,z) \omega^2 =  2 \vol(T(\a,m)).
\]
\end{corollary}

\begin{theorem} \label{Phi-integrable-estimate} 
We have
\[
\int_{\Xa} \abs{\Phi(\a,m,z)} \omega^2 = O(m^2 \log(m))
\]
for large $m$.
\end{theorem}
\begin{proof}
Using decomposition~\eqref{Phi-decomp} and Proposition~\ref{integral-Phi-0} we are left with proving
\[
\int_{\Xa} \abs{ \sum_{n=1}^\infty \Phi_n(\a,m,1,z) } \omega^2 = 2 \vol(T(\a,m)) = O(m^2 \log(m)).
\]
The first equality follows with $\Phi_n(\a,m,1,z) \ge 0$ for $n \in \N$ and Corollary~\ref{Phi-0-integral}, the second one with $\vol(T(\a,m)) = q(\a,m) \zeta_K(-1)/2$ and Corollary~\ref{qL-growth}.
\end{proof}

\begin{corollary} \label{almost-everywhere-coro} 
The generating series
\[
\sum_{m=1}^\infty \Phi(\a,m,z) q^m
\und \sum_{m=1}^\infty \abs{\Phi(\a,m,z) q^m}
\]
with $q \in \C$, $|q|<1$ converge absolutely for almost all $z \in \H^2$ and are integrable over $\Xa$.
\end{corollary}
\begin{proof}
We have by Tonelli's theorem and Theorem~\ref{Phi-integrable-estimate} for appropriate constant $C_1,C_2>0$
\begin{align*}
\int_{\Xa}  \sum_{m=1}^\infty \abs{ \Phi(\a,m,z) q^m} \omega^2
&= \sum_{m=1}^\infty \br{\int_{\Xa} \abs{\Phi(\a,m,z) } \omega^2} |q|^m\\
&\le C_1 + \sum_{m \gg 1}^\infty C_2 m^2 \log(m) |q|^m < \infty.
\end{align*}
This implies all stated assertions.
\end{proof}

\begin{remark} \label{almost-nowhere-continuous} 
By assigning $\infty$ to the values where
\[
\sum_{m=1}^\infty \Phi(\a,m,z) q^m
\]
diverges we can interpret the series as well-defined function $\Xa \to \P^1(\C)$. However, this function is discontinuous at all $z \in \Xa$ where the series converges. This is because the set of singularities coming from the logarithmic singularities of the single $\Phi(\a,m,z)$ at the Hirzebruch--Zagier divisors lies dense in $\Xa$.
\end{remark}

\begin{theorem} \label{modular-integral}  
Assume that $D$ is odd. The integral of
\[
\sum_{m=1}^\infty \Phi(\a,m,z) e(\tau m)
\]
over $\Xa$ is a holomorphic modular form of weight $2$ in $\tau \in \H$ up to a constant.
\end{theorem}
\begin{proof}
By Corollary~\ref{almost-everywhere-coro} the integral of the series is well defined and by Theorem~\ref{Phi-integral} equals to
\[
- \frac{1}{12} \sum_{m=1}^\infty \sigma(\a,m,-1) e(\tau m).
\]
In \cite[Corollary~4.1]{buckEisen} the Fourier expansion of an holomorphic Eisenstein series for $\Gamma_0(D)$ of nebentypus $\chi_D$ and weight $2$ is stated to be
\[
1 + \frac{2}{ L(-1,\chi_D)} \sum_{m=1}^\infty \sigma(\a,m,-1) e(m\tau).
\]
\end{proof}

\end{document}